\documentclass[times]{speauth}

\usepackage{moreverb}

\usepackage{amsfonts,amssymb,newlfont,graphicx,subfigure,algorithmic,algorithm,fixltx2e,booktabs,bm}
\usepackage{epstopdf}
\usepackage{empheq} % autoload amsmath
\usepackage{appendix}
\numberwithin{algorithm}{section}
\usepackage[nomarkers,nolists]{endfloat}

\usepackage[colorlinks,bookmarksopen,bookmarksnumbered,citecolor=red,urlcolor=red]{hyperref}
\numberwithin{algorithm}{section}
\newcommand\BibTeX{{\rmfamily B\kern-.05em \textsc{i\kern-.025em b}\kern-.08em
T\kern-.1667em\lower.7ex\hbox{E}\kern-.125emX}}

\newtheorem{thm}{Theorem}[section]

\newtheorem{rem}{Remark}[section]

\numberwithin{equation}{section}
\renewcommand{\theequation}{\thesection.\arabic{equation}}
\def\simgt{\,\hbox{\lower0.6ex\hbox{$>$}\llap{\raise0.3ex\hbox{$\sim$}}}\,}
\def\simlt{\,\hbox{\lower0.6ex\hbox{$<$}\llap{\raise0.3ex\hbox{$\sim$}}}\,}
\def\simgteq{\,\hbox{\lower0.6ex\hbox{$\ge$}\llap{\raise0.6ex\hbox{$\sim$}}}\,}
\def\simlteq{\,\hbox{\lower0.6ex\hbox{$\le$}\llap{\raise0.6ex\hbox{$\sim$}}}\,}
\makeatletter
\def\user@resume{resume}
\def\user@intermezzo{intermezzo}
\newcounter{previousequation}
\newcounter{lastsubequation}
\newcounter{savedparentequation}
\makeatother

\begin{document}
%%
%\runningheads{Kareem T. Elgindy}{High-Order, Stable, And Efficient Pseudospectral Method}
\runningheads{Kareem T. Elgindy}{Barycentric Gegenbauer Quadratures}

\title{High-Order, Stable, And Efficient Pseudospectral Method Using Barycentric Gegenbauer Quadratures}

\author{Kareem T. Elgindy\corrauth}

\address{Mathematics Department, Faculty of Science, Assiut University, Assiut 71516, Egypt}

\corraddr{Mathematics Department, Faculty of Science, Assiut University, Assiut 71516, Egypt}

\begin{abstract}
The work reported in this article presents a high-order, stable, and efficient Gegenbauer pseudospectral method to solve numerically a wide variety of mathematical models. The proposed numerical scheme exploits the stability and the well-conditioning of the numerical integration operators to produce well-conditioned systems of algebraic equations, which can be solved easily using standard algebraic system solvers. The core of the work lies in the derivation of novel and stable Gegenbauer quadratures based on the stable barycentric representation of Lagrange interpolating polynomials and the explicit barycentric weights for the Gegenbauer-Gauss (GG) points. A rigorous error and convergence analysis of the proposed quadratures is presented along with a detailed set of pseudocodes for the established computational algorithms. The proposed numerical scheme leads to a reduction in the computational cost and time complexity required for computing the numerical quadrature while sharing the same exponential order of accuracy achieved by \cite{Elgindy2013}. The bulk of the work includes three numerical test examples to assess the efficiency and accuracy of the numerical scheme. The present method provides a strong addition to the arsenal of numerical pseudospectral methods, and can be extended to solve a wide range of problems arising in numerous applications.
%\subclass{65D30 \and 65D32 \and 65D05}
\end{abstract}
\keywords{Barycentric interpolation; Gegenbauer polynomials; Gegenbauer quadrature; Integration matrix; Pseudospectral method.}
\maketitle

\vspace{-6pt}

\section{Introduction}
\label{int}
The past few decades have seen a conspicuous attention towards the solution of differential problems by working on their integral reformulations; cf. \cite{Elgindy2009,Elgindy2013a,Franccolin2014,Elgindy2012d,Tang2015,Coutsias1996,Greengard1991,Viswanath2015,Driscoll2010,Olver2013,El-Gendi1969}. Perhaps one of the reasons that laid the foundation of this methodology appears in the well stability and boundedness of numerical integral operators in general whereas numerical differential operators are inherently ill-conditioned; cf. \cite{Funaro1987,Greengard1991,Elgindy2013b}. The numerical integral operator used in the popular pseudospectral methods is widely known as the spectral integration matrix (also called the operational matrix of integration), which dates back to \cite{El-Gendi1969} in the year 1969. In fact, the introduction of the numerical integration matrix has provided the key to apply the rich and powerful matrix linear algebra in many areas \cite{Elgindy2013b}.

In 2013, \cite{Elgindy2013} presented some novel numerical quadratures based on the concept of numerical integration matrices. Their unified approach employed the
Gegenbauer basis polynomials to achieve rapid convergence rates for small/medium range of spectral expansion terms while using Chebyshev and Legendre bases polynomials for a large-scale number of expansion terms. The established quadratures were presented in basis form, and were parameter optimized in the sense of minimizing the Gegenbauer parameter associated with the quadrature truncation error. This key idea allowed for interpolating the integrand function at some Gegenbauer-Gauss (GG) sets of points called the adjoint GG points instead of using the same integration points for constructing the numerical quadrature. This approach provides in turn the luxury of evaluating quadratures for any arbitrary integration points for any desired degree of accuracy; thus increasing the accuracy of collocation schemes using relatively small number of collocation points; cf. \cite{Elgindy2013,Elgindy2013a,Elgindy2012d,Elgindy2016}.

%The questions of finding the optimal Gegenbauer parameters and how to select the corresponding Gegenbauer basis were illustrated earlier by \cite{Elgindy2013}. Moreover, the advantages of using Gegenbauer polynomials as basis functions over Chebyshev and Legendre bases polynomials %including faster convergence rates for small/medium range of the spectral expansion terms 
%were also presented in a number of articles; cf. \cite{Elgindy2013b} for a brief survey on this topic. Therefore, the aforementioned questions are out of the scope of the current article. Instead, 
In the current article, we extend the works of Elgindy and Smith-Miles \cite{Elgindy2013,Elgindy2013a,Elgindy2016}, and develop some novel and efficient Gegenbauer integration matrices (GIMs) and quadratures based on the stable barycentric representation of Lagrange interpolating polynomials and the explicit barycentric weights for the GG points. The present numerical scheme represents an improvement over the aforementioned works as we reduce the computational cost and time complexity required for computing the numerical quadratures while sharing the same order of accuracy.

%The main focus and contribution of the current article is to improve the works of \cite{Elgindy2013,Elgindy2013a} by reducing the computational cost and time complexity required for computing the numerical quadratures while sharing the same order of accuracy. To this end, we introduce some novel and efficient optimal Gegenbauer integration matrices (GIMs) and quadratures based on the stable barycentric representation of Lagrange interpolating polynomials and the explicit barycentric weights for the GG points. 

The rest of the article is organized as follows: In Section \ref{sec:pre}, we give some basic preliminaries relevant to Gegenbauer polynomials and their orthogonal basis and linear barycentric rational interpolations. In Section \ref{sec:BGIMAQ}, we derive the barycentric GIM and quadrature, and provide a rigorous error and convergence analysis. In Section \ref{sec:TOBGAQ1}, we construct the optimal barycentric GIM in some optimality measure, and analyze its associated quadrature error in Section \ref{subsec:ECA123}. Section \ref{sec:CA1} is devoted for a comprehensive discussion on some efficient computational algorithms required for the construction of the novel GIMs and quadratures. A discussion on how to resolve boundary-value problems using the barycentric GIM is presented in Section \ref{subsec: RBVPUTBG1}. Three numerical test examples are studied in Section \ref{sec: NEAA1} to assess the efficiency and accuracy of the numerical scheme. We provide some concluding remarks and possible future directions in Section \ref{conc}. Finally, a detailed set of pseudocodes for the established computational algorithms is presented in Appendix \ref{appendix:CAP}.

\section{Preliminaries}
\label{sec:pre}
\vspace{-2pt}
In this section, we briefly recall some preliminary properties of the Gegenbauer polynomials and their orthogonal interpolations. %Moreover, we review the standard square and optimal rectangular GIM and their associated quadratures recently developed by \cite{Elgindy2013}.
The Gegenbauer polynomial $G_n^{(\alpha )}(x)$, of degree $n \in \mathbb{Z}^+$, and associated with the parameter $\alpha > -1/2$, is a real-valued function, which appears as an eigensolution to a singular Sturm-Liouville problem in the finite domain $[-1, 1]$ \cite{Szego1975}. It is a symmetric Jacobi polynomial, $P_n^{(\nu_1, \nu_2)}(x)$, with $\nu_1 = \nu_2 = \alpha - 1/2$, and can be standardized through \cite[Eq. (A.1)]{Elgindy2012d}. It is an odd function for odd $n$ and an even function for even $n$. The Gegenbauer polynomials can be generated by the three-term recurrence equations \cite[Eq. (A.4)]{Elgindy2013}, or in terms of the hypergeometric functions \cite[Eq. (2.3)]{Elgindy2016}. The weight function for the Gegenbauer polynomials is the even function $w^{(\alpha)}(x) = {(1 - {x^2})^{\alpha - 1/2}}$. The Gegenbauer polynomials form a complete orthogonal basis polynomials in $L_{w^{(\alpha)}}^2[-1, 1]$, and their orthogonality relation is given by the following weighted inner product:
\begin{equation}
\left(G_m^{(\alpha)}, G_n^{(\alpha)}\right)_{w^{(\alpha)}} = \int_{ - 1}^1 {G_m^{(\alpha )}(x)\, G_n^{(\alpha )}(x)\, w^{(\alpha)}(x)\, dx}  = \left\| {G_n^{(\alpha )}} \right\|_{{w^{(\alpha )}}}^2 {\delta _{m,n}} = {\lambda}_n^{(\alpha )} {\delta _{m,n}},
\end{equation}
 where
\begin{equation}\label{sec:pre:eq:normak1}
{\lambda}_n^{(\alpha )} = \left\| {G_n^{(\alpha )}} \right\|_{{w^{(\alpha )}}}^2 = \frac{{{2^{1 - 2\alpha }}\,\pi\,   \Gamma (n + 2\alpha )}}{{n!\,(n + \alpha )\,{\Gamma ^2}(\alpha )}},
\end{equation}
is the normalization factor, and $\delta _{m,n}$ is the Kronecker delta function. We denote the GG nodes and their corresponding Christoffel numbers by $x_{n,k}^{(\alpha)}, \varpi_{n,k}^{(\alpha)}, k = 0, \ldots, n$, respectively. The reader may consult Refs. \cite{Abramowitz1965,Szego1975,Bayin2006,Elgindy2013,Elgindy2013b} for more information about this elegant family of polynomials.
\subsection{Orthogonal Gegenbauer interpolation}
\label{subsec:OGI}
The function
\begin{equation}\label{sec:ort:eq:modal}
{P_n}f(x) = \sum\limits_{j = 0}^n {{{\tilde f}_j}\,G_{j}^{(\alpha )}(x)} ,
\end{equation}
is the Gegenbauer interpolant of a real function $f$ defined on $[-1, 1]$, if we compute the coefficients ${{\tilde f}_j}$ so that
\begin{equation}
	{P_n}f({x_k}) = f({x_k}),\quad k = 0, \ldots ,n,
\end{equation}
for some nodes $x_k \in [-1, 1], k = 0, \ldots ,n$. If we choose the interpolation points $x_k, k = 0, \ldots, n$, to be the GG nodes $x_{n,k}^{(\alpha)}, k = 0, \ldots,n$, then we can simply compute the discrete Gegenbauer transform using the discrete inner product created from the GG quadrature by the following formula:
\begin{equation}\label{sec:ort:eq:sgt}
{{\tilde f}_j} = \frac{{{{\left({P_n}f,G_{j}^{(\alpha )}\right)}_{n}}}}{{\left\| {G_{j}^{(\alpha )}} \right\|_{w^{(\alpha )}}^2}} = \frac{{{{\left(f,G_{j}^{(\alpha )}\right)}_{n}}}}{{\left\| {G_{j}^{(\alpha )}} \right\|_{w^{(\alpha )}}^2}} = \frac{1}{{\lambda _{j}^{(\alpha )}}}\sum\limits_{k = 0}^n {\varpi _{n,k}^{(\alpha )}\,f_{n,k}^{(\alpha )}\,G_{j}^{(\alpha )}\left(x_{n,k}^{(\alpha )}\right)},\quad j = 0, \ldots ,n,
\end{equation}
where $f_{n,k}^{(\alpha )} = f\left(x_{n,k}^{(\alpha )}\right)\, \forall k$. Substituting Equation \eqref{sec:ort:eq:sgt} into \eqref{sec:ort:eq:modal} yields the Lagrange basis form of the Gegenbauer interpolation (nodal approximation) of $f$ at the GG nodes as follows:
\begin{equation}\label{sec:ort:eq:Lagint1}
	{P_n}f(x) = \sum\limits_{k = 0}^n {{f_{n,k}^{(\alpha)}}\,\mathcal{L} _{n,k}^{(\alpha )}(x)} ,
\end{equation}
where $\mathcal{L} _{n,k}^{(\alpha )}(x)$, are the Lagrange interpolating polynomials defined by
\begin{equation}\label{sec:ort:eq:Lag1}
	\mathcal{L} _{n,k}^{(\alpha )}(x) = \varpi _{n,k}^{(\alpha )}\sum\limits_{j = 0}^n {{{\left( {\lambda _{j}^{(\alpha )}} \right)}^{ - 1}}\,G_{j}^{(\alpha )}\left(x_{n,k}^{(\alpha )}\right)\,G_{j}^{(\alpha )}(x)} ,\quad k = 0, \ldots ,n.
\end{equation}

It is noteworthy to mention that the cost of the discrete Gegenbauer transform \eqref{sec:ort:eq:sgt} amounts to $O(n^2)$ operations in general by direct evaluation, except for Chebyshev points, where the cost can be reduced to $O(n\, \log n)$ using the FFT. Therefore, we need to perform $O(n^2)$ operations in general to compute the value of the modal interpolant \eqref{sec:ort:eq:modal} for every new value $x$. Similarly, and despite the numerically stable form of the nodal interpolating polynomial \eqref{sec:ort:eq:Lagint1}, its evaluation also requires $O(n^2)$ operations in general at each new value $x$; cf. \cite{Kopriva2009,Wang2014}. Moreover, adding a new data pair $\left( {x_{n + 1,k}^{(\alpha )},f_{n + 1,k}^{(\alpha )}} \right)$ requires an entirely new computation of every $\mathcal{L} _{n+1,k}^{(\alpha )}(x),\, k = 0, \ldots, n+1$.
\subsection{The linear barycentric rational Lagrange interpolation}
\label{subsec:TBLI}
A fast and efficient variant of Lagrange interpolation is the linear barycentric rational Lagrange interpolation, which gained much attention in recent years \cite{Berrut2004,Higham2004,Berrut2005,Wang2014a,Berrut2014}. The barycentric formula of the Lagrange interpolating polynomial is defined by
\begin{equation}\label{eq:Bary1}
	{\cal L}_{B,n,i}^{(\alpha )}(x) = \frac{{\xi _{n,i}^{(\alpha )}}}{{x - x_{n,i}^{(\alpha )}}}/\sum\limits_{j = 0}^n {\frac{{\xi _{n,j}^{(\alpha )}}}{{x - x_{n,j}^{(\alpha )}}}} ,\quad i = 0, \ldots ,n,
\end{equation}
where
\begin{equation}\label{eq:Bweights1}
	\xi _{n,j}^{(\alpha )} = \frac{1}{{\prod\nolimits_{\scriptstyle i = 0\hfill\atop
\scriptstyle i \ne j\hfill}^n {\left(x_{n,j}^{(\alpha )} - x_{n,i}^{(\alpha )}\right)} }}, \quad j = 0, \ldots ,n,
\end{equation}
are the barycentric weights. The barycentric formula for ${P_n}f$ is therefore defined by
\begin{equation}\label{sec:ort:eq:LagintB1}
	{P_{B,n}}f(x) = \sum\limits_{i = 0}^n {{f_{n,i}^{(\alpha)}}\,\mathcal{L} _{B,n,i}^{(\alpha )}(x)}.
\end{equation}
The linear barycentric rational Lagrange interpolation enjoys several advantages, which makes it very efficient in practice: (i) The barycentric Lagrange interpolating polynomials of Eq. \eqref{eq:Bary1} are scale-invariant; thus avoid any problems of underflow and overflow \cite{Berrut2004}. (ii) They are forward stable for Gauss sets of interpolating points with slowly growing Lebesgue constant \cite{Higham2004}. %In fact, the barycentric formula \eqref{sec:ort:eq:LagintB1} is numerically stable for any set of interpolating points inside the interpolation interval $[-1, 1]$ with a small Lebesgue constant.
(iii) Once the weights are computed, the interpolant at any point $x$ will take only $O(n)$ floating point operations to compute; cf. \cite{Kopriva2009,Gander2005,Berrut2004}. \cite{Berrut2004} have further considered the barycentric Lagrange interpolation to be the `\textit{standard method of polynomial interpolation}.'

   Despite the pleasant features of the barycentric formula discussed above, the direct calculation of the barycentric weights using Eq. \eqref{eq:Bweights1} suffers from significant numerical errors when the number of interpolating points is large, since the differences $\left(x_{n,j}^{(\alpha )} - x_{n,i}^{(\alpha )}\right)$ appearing in the denominator are subject to floating-point cancellation errors for large $n$. Fortunately, the recent works of \cite{Wang2012} and \cite{Wang2014} showed that the barycentric weights for the Gauss points can be expressed explicitly in terms of the corresponding quadrature weights for classical orthogonal polynomials. In particular, for the Gegenbauer polynomials considered in this work, we have the following theorem, which mitigates the harm of cancellation error arising in Eq. \eqref{eq:Bweights1}.
\begin{thm}[\cite{Wang2014}]\label{thm:Wang2014}
The barycentric weights for the GG points are given by
\begin{equation}\label{eq:BGG1}
	\xi _{n,i}^{(\alpha )} = {( - 1)^i} \sqrt {\left( {1 - {{\left( {x_{n,i}^{(\alpha )}} \right)}^2}} \right)\,\varpi _{n,i}^{(\alpha )}},\quad i = 0, \ldots ,n,
\end{equation}
where $\left\{ {x_{n,i}^{(\alpha )},\varpi _{n,i}^{(\alpha )}} \right\}_{i = 0}^n$ is the set of GG points and quadrature weights, respectively.
\end{thm}
Based on a fast $O(n)$ algorithm for the computation of Gaussian quadrature due to \cite{Hale2013}, Theorem \ref{thm:Wang2014} leads to an $O(n)$ computational scheme for the barycentric weights. A MATLAB code for the GG points and quadrature weights can be found in Chebfun ``jacpts'' function; cf. \cite{Trefethen2011}. 

So far, the reader may expect that cancellation errors arising in the calculation of the barycentric weights are eliminated by using the numerically more stable formula \eqref{eq:BGG1}, but the story does not end here. Recall that the GG points cluster quadratically near the endpoints $\pm 1$ as $n \to \infty$; in addition, the positive GG points increase monotonically when $\alpha$ decreases. %; cf. \cite{Area2004}. 
Therefore, Formula \eqref{eq:BGG1} may still suffer from cancellation effects. Fortunately, it is possible in this case to modify the computation and avoid cancellation by introducing the useful transformation $x = \cos(\theta)$, as stated by the following theorem, which gives a more numerically stable expression for calculating the barycentric weights.
\begin{thm}\label{thm:Kimo2016}
The barycentric weights for the GG points are given by
\begin{equation}\label{eq:BGG1Mod1}
\xi _{n,i}^{(\alpha )} = {( - 1)^i}\,\sin \left( {{{\cos }^{ - 1}}\left( {x_{n,i}^{(\alpha )}} \right)} \right)\sqrt {{\mkern 1mu} \varpi _{n,i}^{(\alpha )}} ,\quad i = 0, \ldots ,n.
\end{equation}
\end{thm}

\section{The barycentric GIM and quadrature}
\label{sec:BGIMAQ}
In many problems and applications, one needs to convert the integral equations involved in the mathematical models into algebraic equations. Such procedures also require some expressions for approximating the integral operators involved in the integral equations. In a Gegenbauer collocation method, this operation is conveniently carried out through the GIM; cf. \cite{Elgindy2013a,Elgindy2013b,Elgindy2012c,Elgindy2012d,Elgindy2013,Elgindy2016}. The GIM is simply a linear map which takes a vector of $n$ function values $f(x_i)$ to a vector of $n$ integral values $\int_{ - 1}^{{x_i}} {f(x)\,dx}$, for a certain set of integration nodes ${\{x_i\}_{i=0}^n}$. It represents an easy, stable, and efficient numerical integration operator for approximating the definite integrals of the function $f(x)$ on the intervals $[-1, x_i], i = 0, \ldots, n$, which frequently arise in collocating integral equations, integro-differential equation, ordinary and partial differential equations, optimal control problems, etc.; cf. \cite{Elgindy2009,Elgindy2013a,Elgindy2012c,Elgindy2013,Elgindy2012d,Elgindy2016}. One way to achieve such approximations was designed by \cite{Elgindy2013} via integrating the orthogonal Gegenbauer interpolant \eqref{sec:ort:eq:Lagint1}, and the sought definite integration approximations can be simply expressed in a matrix-vector multiplication; cf. \cite[Theorem 2.1]{Elgindy2013}. \cite{Elgindy2013} have further introduced a method for optimally constructing a rectangular GIM by minimizing the magnitude of the quadrature error in some optimality sense; cf. \cite[Theorem 2.2]{Elgindy2013}. In the sequel, we present a novel numerical scheme considered an improvement over the work of \cite{Elgindy2013} for constructing the GIM and its associated quadrature through the stable barycentric representation of Lagrange interpolating polynomials and the explicit barycentric weights for the GG points.

To construct the barycentric GIM and quadrature, we integrate the orthogonal barycentric Gegenbauer interpolant \eqref{sec:ort:eq:LagintB1} on $[-1, {x_{n,j}^{(\alpha )}}]$, for each $j$ so that
\begin{equation}\label{eq:defBi1}
	\int_{ - 1}^{x_{n,j}^{(\alpha )}} {{P_{B,n}}f(x)\,dx}  = \sum\limits_{i = 0}^n {f_{n,i}^{(\alpha )}{\mkern 1mu} \int_{ - 1}^{x_{n,j}^{(\alpha )}} {{\cal L}_{B,n,i}^{(\alpha )}(x)\,dx} }, \quad j = 0, \ldots, n.
\end{equation}
Introducing the change of variable
\begin{equation}\label{eq:cov1}
	x = \frac{1}{2}\left( {\left( {x_{n,j}^{(\alpha )} + 1} \right)\,t + x_{n,j}^{(\alpha )} - 1} \right),
\end{equation}
allows us to rewrite the definite integrals \eqref{eq:defBi1} further as
\begin{equation}
	\int_{ - 1}^{x_{n,j}^{(\alpha )}} {{P_{B,n}}f(x)\,dx}  = \frac{{x_{n,j}^{(\alpha )} + 1}}{2}\sum\limits_{i = 0}^n {f_{n,i}^{(\alpha )}{\mkern 1mu} \int_{ - 1}^1 {{\cal L}_{B,n,i}^{(\alpha )}\left( {t; - 1,x_{n,j}^{(\alpha )}} \right)\,dt} } ,\quad j = 0, \ldots ,n.
\end{equation}
Since the barycentric Lagrange interpolating polynomials ${{\cal L}_{B,n,i}^{(\alpha )}\left( {t; - 1,x_{n,j}^{(\alpha )}} \right)}$ are polynomials of degree less than or equal to $n$, the integrals $\int_{ - 1}^1 {{\cal L}_{B,n,i}^{(\alpha )}\left( {t; - 1,x_{n,j}^{(\alpha )}} \right)\,dt}$ can be computed exactly using an $\left\lceil {(n + 1)/2} \right\rceil $-point Legendre-Gauss (LG) quadrature, where $\left\lceil  \cdot  \right\rceil$ denotes the ceiling function. In particular, let $N = \left\lceil {(n - 1)/2} \right\rceil, \left\{ {x_{N,k}^{(0.5)}} \right\}_{k = 0}^N$ be the zeros of the $(N+1)$th-degree Legendre polynomial, $L_{N+1}(t)$, and $\left\{ {\varpi _{N,k}^{(0.5)}} \right\}_{k = 0}^N$ be the LG weights defined by
\begin{equation}
	{\varpi _{N,k}^{(0.5)}} = \frac{2}{{\left( {1 - {{\left( {{x_{N,k}^{(0.5)}}} \right)}^2}} \right)\,{{\left[ {{L'_{N + 1}}\left({x_{N,k}^{(0.5)}}\right)} \right]}^2}}},\quad k = 0, \ldots ,N,
\end{equation}
where ${L'_{N + 1}}$ denotes the derivative of ${{{L}_{N + 1}}}$. Then,
\begin{equation}
	\int_{ - 1}^1 {{\cal L}_{B,n,i}^{(\alpha )}\left( {t; - 1,x_{n,j}^{(\alpha )}} \right)\,dt}  = \sum\limits_{k = 0}^N {{\varpi _{N,k}^{(0.5)}}\,{\cal L}_{B,n,i}^{(\alpha )}\left( {{x_{N,k}^{(0.5)}}; - 1,x_{n,j}^{(\alpha )}} \right)} .
\end{equation}
Hence,
\begin{equation}\label{eq:pbquad1}
	\int_{ - 1}^{x_{n,j}^{(\alpha )}} {{P_{B,n}}f(x)\,dx}  = \sum\limits_{i = 0}^n {{p_{B,j,i}^{(1)}}\,f_{n,i}^{(\alpha )}} ,\quad j = 0, \ldots ,n,
\end{equation}
where $p_{B,j,i}^{(1)}, i,j = 0, \ldots ,n$, are the elements of the first-order barycentric GIM given by
\begin{equation}\label{eq:pbGIMe1}
	{p_{B,j,i}^{(1)}} = \frac{{x_{n,j}^{(\alpha )} + 1}}{2}\sum\limits_{k = 0}^N {\;{\varpi _{N,k}^{(0.5)}}\,{\cal L}_{B,n,i}^{(\alpha )}\left( {{x_{N,k}^{(0.5)}}; - 1,x_{n,j}^{(\alpha )}} \right)} ,\quad i,j = 0, \ldots ,n.
\end{equation}
Eqs. \eqref{eq:pbquad1} provide the values of the barycentric Gegenbauer quadrature on the intervals $\left[-1, x_{n,j}^{(\alpha)}\right], j = 0, \ldots, n$, and can be further written in matrix notation as
\begin{equation}\label{eq:pbquadmat1}
{\mathbf{I}}_n^{(\alpha )} = {\mathbf{P}}_B^{(1)}\,{\mathbf{F}},	
\end{equation}
where ${\mathbf{I}}_n^{(\alpha )} = {\left( {\int_{ - 1}^{x_{n,0}^{(\alpha )}} {{P_{B,n}}f(x)\,dx} ,\int_{ - 1}^{x_{n,1}^{(\alpha )}} {{P_{B,n}}f(x)\,dx} , \ldots ,\int_{ - 1}^{x_{n,n}^{(\alpha )}} {{P_{B,n}}f(x)\,dx} } \right)^T},\,{\mathbf{F}} = {\left( {f_{n,0}^{(\alpha )},f_{n,1}^{(\alpha )}, \ldots ,f_{n,n}^{(\alpha )}} \right)^T}$, and ${\mathbf{P}}_B^{(1)} = \left( {{p_{B,j,i}^{(1)}}} \right)$, $i,j = 0, \ldots ,n$ is the first-order barycentric GIM. Clearly, ${\mathbf{P}}_B^{(1)}$ is a square matrix of size $(n + 1)$. Notice also that the barycentric Chebyshev and Legendre matrices can be directly recovered by setting $\alpha = 0; 0.5$, respectively.

Similar to the works of \cite{Elgindy2013} and \cite{Elgindy2016}, the $q$th-order barycentric GIM can be directly generated from the first-order barycentric GIM by the following formulas
\begin{equation}\label{eq:kok111}
	p_{B,j,i}^{(q)} = \frac{{{{\left(x_{n,j}^{(\alpha )} - x_{n,i}^{(\alpha )}\right)}^{q - 1}}}}{{(q - 1)!}}p_{B,j,i}^{(1)},\quad i,j = 0, \ldots ,n,
\end{equation}
or in matrix form,
\begin{equation}\label{eq:kok222}
{{\mathbf{P}}_B}^{(q)} = {\mkern 1mu} \frac{1}{{(q - 1)!}}\left( {\left( {{\mathbf{x}}_n^{(\alpha )} \otimes {{\mathbf{J}}_{1,n + 1}}} \right) - \left( {{{\left( {{\mathbf{x}}_n^{(\alpha )}} \right)}^T} \otimes {{\mathbf{J}}_{n + 1,1}}} \right)} \right)_{(q-1)} \circ {{\mathbf{P}}_B}^{(1)},
\end{equation}
where ${\mathbf{x}}_n^{(\alpha )} = {[x_{n,0}^{(\alpha )},x_{n,1}^{(\alpha )}, \ldots ,x_{n,n}^{(\alpha )},]^T}, {\mathbf{\mathcal{A}}_{(m)}} = \underbrace {\mathbf{\mathcal{A}} \circ \mathbf{\mathcal{A}} \circ  \ldots  \circ \mathbf{\mathcal{A}}}_{m - {\text{times}}}$, for any $\mathbf{\mathcal{A}} \in {\mathbb{R}^{l \times l}}, m, l \in {\mathbb{Z}^ + }, {{\mathbf{J}}_{i,j}}$ is the all ones matrix of size $i \times j$, ``$\otimes$'' and ``$\circ$'' denote the Kronecker product and Hadamard product (entrywise product), respectively; cf. \cite[Eq. (4.42)]{Elgindy2016}. On the interval $[0, 1]$, the R.H.S. of each of Eqs. \eqref{eq:kok111} and \eqref{eq:kok222} is divided by $2^q$.
\subsection{Error and Convergence Analysis}
\label{subsec:EACA1}
Since the barycentric formula of the Lagrange interpolating polynomial \eqref{eq:Bary1} is mathematically equivalent to the standard Lagrange interpolating polynomials defined by \eqref{sec:ort:eq:Lag1}, the established barycentric GIM and quadrature share the same order of error and convergence properties of the GIM and quadrature developed by \cite{Elgindy2013}; therefore, the following theorem is straightforward.
\begin{thm}\label{sec1:thm:elhawarykareem1}
Let $\mathbb{S}_n^{(\alpha )} = \left\{ {x_{n,j}^{(\alpha )},\;j = 0, \ldots ,n} \right\}, n \in \mathbb{Z}^+$, be the set of GG points. Moreover, let $f(x) \in C^{n+1}[-1, 1]$ be approximated by the barycentric Gegenbauer expansion series \eqref{sec:ort:eq:LagintB1}. Then there exist some numbers $\zeta_{n,j}^{(\alpha)} \in [-1, 1], j = 0, \ldots, n$ such that
\begin{equation}\label{sec1:eq:ultraint2kimohat}
\int_{ - 1}^{{x_{n,j}^{(\alpha )}}} {f(x)  dx}  = \sum\limits_{i = 0}^n {p_{B,j,i}^{(1)}} {f_{n,i}^{(\alpha )}} + E_n^{(\alpha )}\left({x_{n,j}^{(\alpha )}},{\zeta_{n,j}^{(\alpha)}}\right)\, \forall x_{n,j}^{(\alpha )} \in \mathbb{S}_n^{(\alpha )},
\end{equation}
where ${p_{B,j,i}^{(1)}}, i,j = 0, \ldots ,n$ are the elements of the first-order barycentric GIM, ${\mathbf{P}}_B^{(1)}$, as defined by Eqs. \eqref{eq:pbGIMe1},
\begin{equation}\label{sec1:eq:errorkimohat}
E_n^{(\alpha )}\left({{x_{n,j}^{(\alpha )}}},{\zeta_{n,j}^{(\alpha)}}\right) = \frac{{{f^{(n + 1)}}\left({\zeta_{n,j}^{(\alpha)}}\right)}}{{(n + 1)!  K_{n + 1}^{(\alpha )}}}  \int_{ - 1}^{{x_{n,j}^{(\alpha )}}} {G_{n + 1}^{(\alpha )}(x)\, dx},
\end{equation}
is the Gegenbauer quadrature error term, and $K_n^{(\alpha)}$ is the leading coefficient of the $n$th-degree Gegenbauer polynomial $G_n^{(\alpha)}(x)$ as defined by \cite[Eq. (A.8)]{Elgindy2013}.
\end{thm}

The following theorem gives the error bounds of the barycentric Gegenbauer quadrature.
\begin{thm}[Error bounds]\label{sec1:thm:krooma1}
Assume that $f(x) \in C^{n+1}[-1, 1]$, and ${\left\| {{f^{(n + 1)}}} \right\|_{{L^\infty }[-1,1]}} \le A \in {\mathbb{R}^ + }$, for some number $n \in \mathbb{Z}_0^+$, where the constant $A$ is independent of $n$. Moreover, let $\int_{-1}^{{x_{n,j}^{(\alpha)}}} {f(x)\,dx}$, be approximated by the barycentric Gegenbauer quadrature \eqref{eq:pbquad1} up to the $(n+1)$th Gegenbauer quadrature expansion term, for each node $x_{n,j}^{(\alpha)}, j = 0, \ldots, n$. Then there exist some positive constants $D_1^{(\alpha)}$ and $D_2^{(\alpha)}$, independent of $n$ such that the truncation error of the barycentric Gegenbauer quadrature, ${E_n^{(\alpha )}\left( {x_{n,j}^{(\alpha )},\zeta _{n,j}^{(\alpha )}} \right)}$, is bounded by the following inequalities:
\begin{empheq}[left={\left| {E_n^{(\alpha )}\left( {x_{n,j}^{(\alpha )},\zeta _{n,j}^{(\alpha )}} \right)} \right| \le}\empheqlbrace]{align}
&{\frac{{A{2^{ - n}} \left( {x_{n,j}^{(\alpha )} + 1} \right) \Gamma \left( {\alpha  + 1} \right) \Gamma \left( {n + 2\alpha  + 1} \right)}}{{\Gamma \left( {2\alpha  + 1} \right)\Gamma \left( {n + 2} \right)\Gamma \left( {n + \alpha  + 1} \right)}},\quad n \ge 0 \wedge \alpha  \ge 0,}\\
	&{\frac{{A{2^{ - n - 1}}\left( {x_{n,j}^{(\alpha )} + 1}  \right) \Gamma \left( \alpha  \right)}}{{\Gamma \left( {n + \alpha  + 1} \right)}} {\left( {\begin{array}{*{20}{c}}
	{\frac{{n - 1}}{2} + \alpha }\\
	{\frac{{n + 1}}{2}}
	\end{array}} \right)},\quad \frac{{n + 1}}{2} \in {\mathbb{Z}^ + } \wedge  - \frac{1}{2} < \alpha  < 0,}
\end{empheq}
\begin{equation}
\left| {E_n^{(\alpha )}\left( {x_{n,j}^{(\alpha )},\zeta _{n,j}^{(\alpha )}} \right)} \right| < {\frac{{A{2^{ - n}} \left( {x_{n,j}^{(\alpha )} + 1} \right) \Gamma \left( \alpha + 1 \right) }}{{\sqrt {\left( {n + 1} \right)\left( {2\alpha  + n + 1} \right)} \Gamma \left( {n + \alpha  + 1} \right)}} {\left( {\begin{array}{*{20}{c}}
	{\frac{n}{2} + \alpha }\\
	{\frac{n}{2}}
	\end{array}} \right)},\quad \frac{n}{2} \in \mathbb{Z}_0^ +  \wedge  - \frac{1}{2} < \alpha  < 0},
\end{equation}
%%%\begin{equation}
	%%%\left| {E_n^{(\alpha )}\left( {x_{n,j}^{(\alpha )},\zeta _{n,j}^{(\alpha )}} \right)} \right| \le \frac{{A{2^{ - n}}\Gamma \left( {\alpha  + 1} \right)\left( {x_{n,j}^{(\alpha )} + 1} \right)\Gamma \left( {n + 2\alpha  + 1} \right)}}{{\Gamma \left( {2\alpha  + 1} \right)\Gamma \left( {n + 2} \right)\Gamma \left( {n + \alpha  + 1} \right)}}\left( {\left\{ \begin{array}{l}
	%%%1,\quad n \ge 0 \wedge \alpha  \ge 0,\\
	%%%\frac{{\left| {\left( {\begin{array}{*{20}{c}}
	%%%{\frac{{n + 1}}{2} + \alpha  - 1}\\
	%%%{\frac{{n + 1}}{2}}
	%%%\end{array}} \right)} \right|\left( {n + 1} \right)!\Gamma \left( {2\alpha } \right)}}{{\Gamma \left( {n + 2\alpha  + 1} \right)}},\quad \frac{{n + 1}}{2} \in {\mathbb{Z}^ + } \wedge  - \frac{1}{2} < \alpha  < 0,\\
	%%%\frac{{2\left| \alpha  \right|\left| {\left( {\begin{array}{*{20}{c}}
	%%%{\frac{n}{2} + \alpha }\\
	%%%{\frac{n}{2}}
	%%%\end{array}} \right)} \right|\left( {n + 1} \right)!\Gamma \left( {2\alpha } \right)}}{{\sqrt {\left( {n + 1} \right)\left( {n + 2\alpha  + 1} \right)} \Gamma \left( {n + 2\alpha  + 1} \right)}},\quad \frac{n}{2} \in \mathbb{Z}_0^ +  \wedge  - \frac{1}{2} < \alpha  < 0
	%%%\end{array} \right.} \right).
%%%\end{equation}
%
Moreover, as $n \to \infty$, the truncation error of the barycentric Gegenbauer quadrature is asymptotically bounded by
\begin{empheq}[left={\left| {E_n^{(\alpha )}\left( {x_{n,j}^{(\alpha )},\zeta _{n,j}^{(\alpha )}} \right)} \right| \simlteq}\empheqlbrace]{align}
&B_1^{(\alpha )}{\left( {\frac{e}{2}} \right)^n}\left( {x_{n,j}^{(\alpha )} + 1} \right){n^{\alpha  - n - \frac{3}{2}}},\quad \alpha  \ge 0,\\
&B_2^{(\alpha )}{\left( {\frac{e}{2}} \right)^n}\left( {x_{n,j}^{(\alpha )} + 1} \right){n^{ - n - \frac{3}{2}}},\quad - \frac{1}{2} < \alpha  < 0,
\end{empheq}
for all $j = 0, \ldots, n$, where $B_1^{(\alpha)} = A D_1^{(\alpha)}$ and $B_2^{(\alpha)} = B_1^{(\alpha)} D_2^{(\alpha)}$.
\end{thm}
\begin{proof}
The proof follows readily from \cite[Theorem 4.3]{Elgindy2016}.
\end{proof}
Clearly, the barycentric Gegenbauer quadrature converges exponentially exhibiting spectral accuracy, since the error decays at a rate faster than any fixed power in $1/n$. %\cite{Elgindy2013} have further proved that the asymptotic truncation error of the quadrature is minimized in the Chebyshev norm exactly at $\alpha = 0$; that is, when
%applying the Chebyshev basis polynomials.
%
\section{The optimal barycentric GIM and quadrature}
\label{sec:TOBGAQ1}
To construct optimal barycentric GIM and quadrature, we follow the approach pioneered by \cite{Elgindy2013}, and seek to determine the optimal Gegenbauer parameter $\alpha_j^*$, which minimizes the magnitude of the quadrature error $E_{n}^{(\alpha )}\left(x_{j},\zeta_{j} \right)$, at any arbitrary node $x_j \in [-1, 1]$, for each $j = 0, \ldots, n$. In particular, The values of the optimal Gegenbauer parameters $\alpha_j^*$ can be determined by solving the following one-dimensional minimization problems:
\begin{equation}\label{subsec:opt:reducedprob1}
{\text{Find }}\alpha _j^* = \mathop {{\text{argmin}}}\limits_{\alpha  >  - 1/2} \eta _{j,n}^2(\alpha ),\quad j = 0, \ldots, n,
\end{equation}
where,
\begin{equation}\label{eq:etakk1}
	\eta _{j,n} {(\alpha )} = \frac{{{2^n}}}{{K_{n + 1}^{(\alpha )}}}\int_{-1}^{{x_{j}}} {G_{n + 1}^{(\alpha )}(x)\,dx}.
\end{equation}
Problems (\ref{subsec:opt:reducedprob1}) can be further converted into unconstrained one-dimensional minimization problems using the change of variable defined by \cite[Eq. (4.17)]{Elgindy2016}. Let $z_{m,j,i}^{(\alpha_j^*)}, j = 0, \ldots, n; i = 0, \ldots, m,$ be the adjoint GG nodes as defined by \cite{Elgindy2013}, for some $m \in \mathbb{Z}^+$; i.e. the zeros of the $(m+1)$th-degree Gegenbauer polynomial, ${G_{m + 1}^{(\alpha_j^* )}(x)}$, for each $j$. The following theorem lays the foundation for deriving the optimal barycentric Lagrange interpolating polynomials of a real-valued function $f$.
\begin{thm}[Optimal barycentric Lagrange interpolating polynomials]
Let $\varpi _{m,k,i}^{(\alpha_k^* )}, k = 0, \ldots, n; i = 0, \ldots, m$ be the set of quadrature weights associated with the adjoint GG points $z_{m,k,i}^{(\alpha_k^*)}, k = 0, \ldots, n; i = 0, \ldots, m$. The functions $\mathcal{L} _{OB,m,i}^{(\alpha_k^*)}(x), i = 0, \ldots, m; k = 0, \ldots ,n,$ defined by
\begin{equation}\label{eq:LOB1}
	{\mathcal{L}}_{OB,m,i}^{(\alpha _k^*)}(x) = \frac{{\xi _{m,k,i}^{(\alpha _k^*)}}}{{x - z_{m,k,i}^{(\alpha _k^*)}}}/\sum\limits_{j = 0}^m {\frac{{\xi _{m,k,j}^{(\alpha _k^*)}}}{{x - z_{m,k,j}^{(\alpha _k^*)}}}} ,\quad i = 0, \ldots ,m;\,k = 0, \ldots ,n,
\end{equation}
are the barycentric Lagrange interpolating polynomials of a real-valued function $f$ constructed through Gegenbauer interpolations at the adjoint GG nodes $z_{m,k,i}^{(\alpha_k^*)}, k = 0, \ldots, n; i = 0, \ldots, m$, with the barycentric weights
\begin{equation}\label{eq:BWeightsk1}
	\xi _{m,k,i}^{(\alpha _k^*)} = {( - 1)^i}\sin \left( {{{\cos }^{ - 1}}\left( {z_{m,k,i}^{(\alpha _k^*)}} \right)} \right)\sqrt {{\mkern 1mu} \varpi _{m,k,i}^{(\alpha _k^*)}},\quad i = 0, \ldots ,m;\,k = 0, \ldots ,n. 
\end{equation}
\end{thm}
\begin{proof}
Denote $f\left(z_{m,k,i}^{(\alpha_k^*)} \right)$ by $f_{m,k,i}^{(\alpha_k^*)}$, for each $k, i$. The classical Lagrange forms of the polynomials of degrees $n$ that interpolate the function $f$ at the set of points $\left\{\left(z_{m,k,i}^{(\alpha_k^*)}, f_{m,k,i}^{(\alpha_k^*)} \right)\right\}, k = 0, \ldots, n; i = 0, \ldots, m$, are defined by
\begin{equation}
	{P_{k,m}f}(x) = \sum\limits_{j = 0}^m {f_{m,k,j}^{(\alpha _k^*)}\,{\mathcal{L}}_{m,j}^{(\alpha _k^*)}(x)} ,
\end{equation}
where ${\mathcal{L}}_{m,j}^{(\alpha _k^*)}(x)$ are the classical Lagrange interpolating polynomials given by
\begin{equation}
	{\mathcal{L}}_{m,j}^{(\alpha _k^*)}(x) = \prod\limits_{\scriptstyle i = 0\hfill\atop
\scriptstyle i \ne j\hfill}^m {\frac{{x - z_{m,k,i}^{(\alpha _k^*)}}}{{z_{m,k,j}^{(\alpha _k^*)} - z_{m,k,i}^{(\alpha _k^*)}}}} .
\end{equation}
Clearly, ${\mathcal{L}}_{m,j}^{(\alpha _k^*)}\left( {z_{m,k,i}^{(\alpha _k^*)}} \right) = {\delta _{i,j}},$ where ${\delta _{i,j}}$ is the Kronecker delta function. Now rewrite the classical Lagrange interpolant in the so-called ``modified Lagrange interpolant'' given by
\begin{equation}
	{P_{k,m}f}(x) = {\psi _k}(x)\sum\limits_{j = 0}^m {f_{m,k,j}^{(\alpha _k^*)}\,\frac{{\xi _{m,k,j}^{(\alpha _k^*)}}}{{x - z_{m,k,j}^{(\alpha _k^*)}}}} ,
\end{equation}
where,
\begin{equation}
	{\psi _k}(x) = \prod\limits_{i = 0}^m {\left( {x - z_{m,k,i}^{(\alpha _k^*)}} \right)},
\end{equation}
and
\begin{equation}
	\xi _{m,k,j}^{(\alpha _k^*)} = \frac{1}{{\prod\limits_{\scriptstyle i = 0\hfill\atop
\scriptstyle i \ne j\hfill}^m {\left( {z_{m,k,j}^{(\alpha _k^*)} - z_{m,k,i}^{(\alpha _k^*)}} \right)} }}.
\end{equation}
Since the function values ${f_{m,k,j}^{(\alpha _k^*)}} = 1$ are evidently interpolated by ${P_{k,m}f}(x) = 1$, we have
\begin{equation}
	{\psi _k}(x)\sum\limits_{j = 0}^m {\frac{{\xi _{m,k,j}^{(\alpha _k^*)}}}{{x - z_{m,k,j}^{(\alpha _k^*)}}}}  = 1;
\end{equation}
hence,
\begin{equation}\label{eq:BLintk18}
{P_{k,m}f}(x) = \frac{{\sum\limits_{i = 0}^m {f_{m,k,i}^{(\alpha _k^*)}\,\frac{\displaystyle {\xi _{m,k,i}^{(\alpha _k^*)}}}{\displaystyle {x - z_{m,k,i}^{(\alpha _k^*)}}}} }}{{\sum\limits_{j = 0}^m {\,\frac{\displaystyle {\xi _{m,k,j}^{(\alpha _k^*)}}}{\displaystyle {x - z_{m,k,j}^{(\alpha _k^*)}}}} }} = \sum\limits_{i = 0}^m {f_{m,k,i}^{(\alpha _k^*)}\,{\mathcal{L}}_{OB,m,i}^{(\alpha _k^*)}(x)} .
\end{equation}
Eq. \eqref{eq:BWeightsk1} follows directly from Theorem \ref{thm:Kimo2016}.
\end{proof}

Integrating Eq. \eqref{eq:BLintk18} on $[-1, x_k]$, and applying the change of variable
\begin{equation}
	x = \frac{1}{2}\left( {\left( {x_{k} + 1} \right)\,t + x_{k} - 1} \right),
\end{equation}
yields,
\begin{equation}\label{eq:KRTS2015}
	\int_{ - 1}^{{x_k}} {{P_{k,m}}f(x)\,dx}  = \frac{{{x_k} + 1}}{2}\sum\limits_{i = 0}^m {f_{m,k,i}^{(\alpha _k^*)}\,\int_{ - 1}^1 {{\mathcal{L}}_{OB,m,i}^{(\alpha _k^*)}(t; - 1,{x_k})\,} \,dt} .
\end{equation}
Hence, the optimal barycentric Gegenbauer quadrature,
\begin{equation}\label{eq:obgq18}
\int_{ - 1}^{{x_k}} {{P_{k,m}}f(x)\,{\mkern 1mu} dx}  = \sum\limits_{i = 0}^m {p_{OB,k,i}^{(1)}{\mkern 1mu} f_{m,k,i}^{(\alpha _k^*)}},
\end{equation}
can be exactly calculated from Eq. \eqref{eq:KRTS2015} using an $(M + 1)$-point LG quadrature, where $p_{OB,k,i}^{(1)}, k = 0, \ldots, n; i = 0, \ldots, m$, are the elements of the first-order optimal barycentric GIM denoted by $\mathbf{P}_{OB}^{(1)}$. The $q$th-order optimal barycentric GIM can be directly generated from the first-order optimal barycentric GIM analogous to \cite[Eq. (2.34)]{Elgindy2013} and \cite[Eq. (4.43)]{Elgindy2016} by the following formulas:
\begin{equation}
	p_{OB,j,i}^{(q)} = \frac{{{{\left(x_{n,j}^{(\alpha )} - z_{m,j,i}^{(\alpha_j^* )}\right)}^{q - 1}}}}{{(q - 1)!}}p_{OB,j,i}^{(1)},\quad j = 0, \ldots ,n; i = 0, \ldots, m.
\end{equation}
\begin{rem}
Although the barycentric GIM is a square, dense matrix that generally leads to dense linear algebra, the optimal barycentric GIM on the other hand is a rectangular, dense matrix that could significantly reduce the computational cost of the collocation scheme for large collocation points; cf. \cite[Remark 5.2]{Elgindy2016b}. 
\end{rem}
\subsection{Error and convergence analysis}
\label{subsec:ECA123}
Now we are ready to present the following useful theorem, which outlines the creation of the optimal barycentric GIM and its associated quadrature. Moreover, the theorem marks the truncation error of the optimal barycentric Gegenbauer quadrature.
\begin{thm}\label{subsec1:theoremnew}
\label{subsec1:kimobasha1}
Let $\mathbb{T}_{n,m} = \{ z_{m,k,i}^{(\alpha _k^*)} ,k = 0,   \ldots ,n; i = 0,  \ldots ,m\}, n,m \in \mathbb{Z}^+$, be the set of adjoint GG points, where $\alpha _k^*$ are the optimal Gegenbauer parameters in the sense that
\begin{equation}\label{sec1:eq:optgegepar1}
\alpha _k^* = \mathop {{\text{argmin}}}\limits_{\alpha  >  - 1/2}  \eta _{k,m}^2(\alpha ),\quad k = 0, \ldots, n,
\end{equation}
and $\eta _{k,m}(\alpha )$ is as defined by Eq. \eqref{eq:etakk1}. Moreover, let $M = \left\lceil {(m - 1)/2} \right\rceil$, and denote by $\left\{ x_{M,s}^{(0.5)},\varpi _{M,s}^{(0.5)} \right\}_{s = 0}^M$, the set of LG points and quadrature weights, respectively. Assume further that $f(x) \in {C^{m+1} }[ - 1,  1]$ is approximated by the Gegenbauer polynomials expansion series such that the Gegenbauer coefficients are computed by interpolating the function $f(x)$ at the adjoint GG points $z_{m,k,i}^{(\alpha _k^*)} \in \mathbb{T}_{n, m}\, \forall k,i$. Then for any arbitrary nodes $x_k \in [-1, 1], k = 0, \ldots, n$, there exist a matrix $\mathbf{P}_{OB}^{(1)} = \left(p_{OB,k,i}^{(1)}\right), k = 0, \ldots, n; i = 0, \ldots, m$, and some numbers $\zeta_k \in [-1, 1], k = 0, \ldots, n$, such that
\begin{equation}\label{sec1:eq:ultraint2kimo}
\int_{ - 1}^{{x_k}} {f(x)\, dx}  = \sum\limits_{i = 0}^m {{p_{OB,k,i}^{(1)}}\,f_{m,k,i}^{(\alpha_k^*)}} + E_m^{(\alpha_k^* )}\left({x_{k}},{\zeta_{k}}\right),
\end{equation}
where
\begin{equation}\label{sec1:eq:qentrieskimo}
p_{OB,k,i}^{(1)} = \frac{{{x_k} + 1}}{2}\sum\limits_{s = 0}^M {{\varpi _{M,s}^{(0.5)}}\,\mathcal{L}_{OB,m,i}^{(\alpha _k^*)}\left( {x_{M,s}^{(0.5)}; - 1,{x_k}} \right)};
\end{equation}
\begin{equation}\label{sec1:eq:errorkimo}
E_m^{(\alpha _k^*)}({x_k},{\zeta _k}) = \frac{{{f^{(m + 1)}}({\zeta _k})}}{{{2^m}\,(m + 1)!}}{\mkern 1mu} {\eta _{k,m}}(\alpha _k^*).
\end{equation}
\end{thm}
\begin{proof}
The quadrature error term \eqref{sec1:eq:errorkimo} follows directly from Theorem \ref{sec1:thm:elhawarykareem1} by substituting the value of $\alpha$ with $\alpha_k^*$, and expanding the Gegenbauer expansion series up to the $(m + 1)$th term.
\end{proof}

The following theorem is a direct corollary of Theorem \ref{sec1:thm:krooma1} and \cite[Theorem 2.3]{Elgindy2013}, and gives the error bounds of the optimal barycentric Gegenbauer quadrature.
\begin{thm}[Error bounds]\label{sec1:thm:krooma118}
Assume that $f(x) \in C^{m+1}[-1, 1]$, and ${\left\| {{f^{(m + 1)}}} \right\|_{{L^\infty }[-1,1]}} \le A \in {\mathbb{R}^ + }$, for some number $m \in \mathbb{Z}_0^+$, where the constant $A$ is independent of $m$. Moreover, let $\int_{-1}^{{x_{k}}} {f(x)\,dx}$, be approximated by the optimal barycentric Gegenbauer quadrature \eqref{eq:obgq18} up to the $(m+1)$th Gegenbauer quadrature expansion term, for each arbitrary integration node $x_{k} \in [-1, 1], k = 0, \ldots, m$. Then there exist some positive constants $D_1^{(\alpha_k^*)}$ and $D_2^{(\alpha_k^*)}$, independent of $m$ such that the truncation error of the barycentric Gegenbauer quadrature, ${E_m^{(\alpha_k^* )}\left( {x_{k},\zeta _{k}} \right)}$, is bounded by the following inequalities:
\small{\begin{empheq}[left={\left| {E_m^{({\alpha_k^*} )}\left( {x_{k},\zeta _{k}} \right)} \right| \le}\empheqlbrace]{align}
&{\frac{{A{2^{ - m}} \left( {x_{k} + 1} \right) \Gamma \left( {{\alpha_k^*}  + 1} \right) \Gamma \left( {m + 2{\alpha_k^*}  + 1} \right)}}{{\Gamma \left( {2{\alpha_k^*}  + 1} \right)\Gamma \left( {m + 2} \right)\Gamma \left( {m + {\alpha_k^*}  + 1} \right)}},\quad m \ge 0 \wedge {\alpha_k^*}  \ge 0,}\\
	&{\frac{{A{2^{ - m - 1}}\left( {x_{k} + 1} \right) \Gamma \left( {\alpha_k^*}  \right)}}{{\Gamma \left( {m + {\alpha_k^*}  + 1} \right)}} {\left( {\begin{array}{*{20}{c}}
	{\frac{{m - 1}}{2} + {\alpha_k^*} }\\
	{\frac{{m + 1}}{2}}
	\end{array}} \right)},\quad \frac{{m + 1}}{2} \in {\mathbb{Z}^ + } \wedge  - \frac{1}{2} < {\alpha_k^*}  < 0,}\\
	&{\frac{{A{2^{ - m}} \left( {x_{k} + 1} \right) \Gamma \left( {\alpha_k^*} + 1  \right) }}{{\sqrt {\left( {m + 1} \right)\left( {2{\alpha_k^*}  + m + 1} \right)} \Gamma \left( {m + {\alpha_k^*}  + 1} \right)}} {\left( {\begin{array}{*{20}{c}}
	{\frac{m}{2} + {\alpha_k^*} }\\
	{\frac{m}{2}}
	\end{array}} \right)},\quad \frac{m}{2} \in \mathbb{Z}_0^ +  \wedge  - \frac{1}{2} < {\alpha_k^*}  < 0}.
\end{empheq}
}
Moreover, as $m \to \infty$, the truncation error of the optimal barycentric Gegenbauer quadrature is asymptotically bounded by
\begin{empheq}[left={\left| {E_m^{({\alpha_k^*} )}\left( {x_{k},\zeta _{k}} \right)} \right| \simlteq}\empheqlbrace]{align}
&B_1^{({\alpha_k^*} )}{\left( {\frac{e}{2}} \right)^m}\left( {x_{k} + 1} \right){m^{{\alpha_k^*}  - m - \frac{3}{2}}},\quad {\alpha_k^*}  \ge 0,\\
&B_2^{({\alpha_k^*} )}{\left( {\frac{e}{2}} \right)^m}\left( {x_{k} + 1} \right){m^{ - m - \frac{3}{2}}},\quad - \frac{1}{2} < {\alpha_k^*}  < 0,
\end{empheq}
for all $k = 0, \ldots, m$, where $B_1^{({\alpha_k^*})} = A D_1^{({\alpha_k^*})}$ and $B_2^{({\alpha_k^*})} = B_1^{({\alpha_k^*})} D_2^{({\alpha_k^*})}$.
\end{thm}

Notice here that Theorem \ref{sec1:thm:krooma118} gives more tight asymptotic error bounds than that obtained in \cite[Theorem 2.3]{Elgindy2013} by realizing that $(m + 1)! = (m + 1) \cdot m! \approx m \cdot \sqrt {2\,\pi \,m} \,{\left( {m/e} \right)^m} = \sqrt {2\,\pi } \,{m^{3/2}}\,{\left( {m/e} \right)^m}$, as $m \to \infty$.

The following theorem parallels \cite[Theorem 2.4]{Elgindy2013}, as it shows that the optimal barycentric Gegenbauer quadrature converges to the optimal Chebyshev quadrature in the $L^{\infty}$-norm, for a large-scale number of expansion terms.
\begin{thm}[Convergence of the optimal barycentric Gegenbauer quadrature]\label{sec1:thm:krooma2}
Assume that $f(x) \in C^{m+1}[-1, 1]$, and $\mathop {{\max}}\nolimits_{\left| x \right| \le 1} \left| {{f^{(m + 1)}}(x)} \right| \le A \in {\mathbb{R}^ + }$, for some number $m \in \mathbb{Z}^+$, where the constant $A$ is independent of $m$. Moreover, let $\int_{ - 1}^{{x_k}} {f(x)\,dx}$ be approximated by the optimal barycentric Gegenbauer quadrature \eqref{eq:obgq18} up to the $(m+1)$th Gegenbauer quadrature expansion term, for each arbitrary integration node $x_k, k = 0, \ldots, m$. Then the optimal barycentric Gegenbauer quadrature converges to the barycentric Chebyshev quadrature in the $L^{\infty}$-norm as $m \to \infty$; that is,
\begin{equation}\label{eq:obgq182}
\sum\limits_{i = 0}^m {p_{OB,k,i}^{(1)}f_{m,k,i}^{(\alpha _k^*)}} \to \sum\limits_{i = 0}^m {p_{B,k,i}^{(1)}f_{m,i}^{(0)}} ,\quad {\text{as }}m \to \infty \;\forall k.
\end{equation}
\end{thm}
\section{Computational algorithms}
\label{sec:CA1}
In this section, we discuss $10$ computational algorithms created to efficiently calculate the developed barycentric GIMs and quadratures and their optimal partners. We commence our discussion with Algorithms \ref{sec:COTBGFTGSOP1:alg1matrix} and \ref{sec:COTBGQFTGSOP1:alg2matrix}, which represent two simple and fast computational algorithms for the computation of the square barycentric GIM and its associated quadrature for the GG set of points; cf. Appendix \ref{appendix:CAP}. We find that one of the major advantages of the established algorithms lies in the reduction of the operational cost required for calculating the GIM and quadrature. Indeed, the barycentric weights can be computed in $O(n)$ operations, whereas the LG points and quadrature weights, $\left\{ {x_{N,i}^{(0.5)},\varpi _{N,i}^{(0.5)}} \right\}_{i = 0}^N$, require $O(N)$ operations. Therefore, the computation of $\left\{ {p_{B,j,i}^{(1)}} \right\}_{i = 0}^n$ through Eq. \eqref{eq:pbGIMe1} costs $O(N \cdot n)$ operations per point-- the same cost required for the evaluation of the barycentric Gegenbauer quadrature through Eq. \eqref{eq:pbquad1} for each point. This amounts to $O\left(N \cdot n^2\right)$ for the construction of the barycentric GIM, ${\mathbf{P}}_B^{(1)}$. On the other hand, the evaluation of the Gegenbauer quadrature derived in \cite[Theorem 2.1]{Elgindy2013} in basis form requires $O\left(n^2\right)$ operations per point while the cost of constructing the associated basis GIM rises up to $O\left(n^3\right)$ operations. Figure \ref{Comp1} shows the average elapsed CPU time in $10$ runs required for the construction of the basis GIM, $\hat {\mathbf{P}}^{(1)}$, derived by \cite{Elgindy2013} and the barycentric GIM, $\mathbf{P}_B^{(1)}$ using $n = 20, 60, 80, 120, 140$ points and $\alpha = -0.25(0.25)2$. Both GIMs were constructed in each case using the same inputs of GG points and quadrature weights, $\left\{ {x_{n,i}^{(\alpha )},\varpi _{n,i}^{(\alpha )}} \right\}_{i = 0}^n$. Clearly, the construction of $\mathbf{P}_B^{(1)}$ is faster than $\hat {\mathbf{P}}^{(1)}$, and the gap grows wider for increasing values of $n$.

\begin{figure}[ht]
\centering
\includegraphics[scale=0.35]{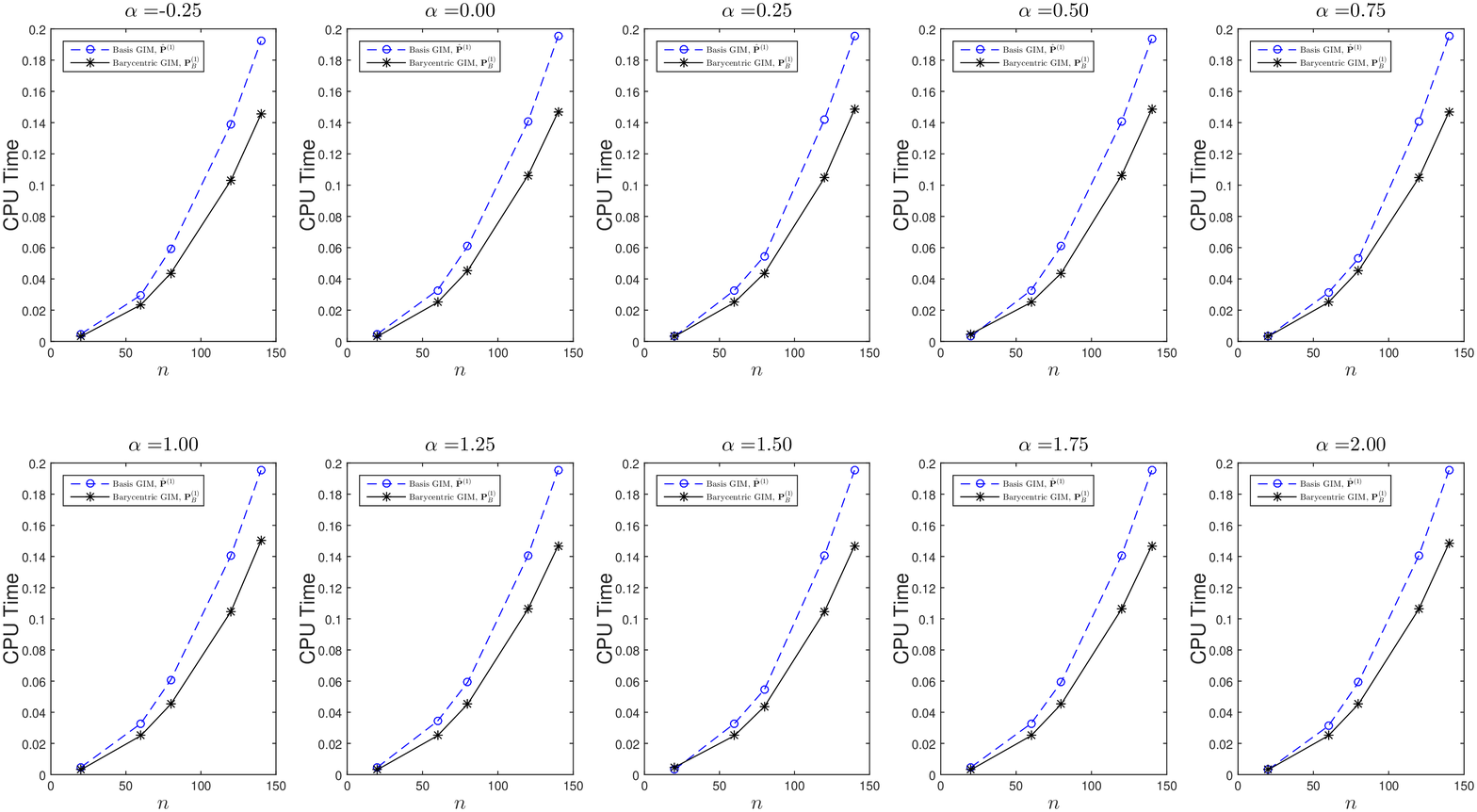}
\caption{The average elapsed CPU time (in seconds) in 10 runs required for the construction of the basis GIM, $\hat {\mathbf{P}}^{(1)}$, derived by \cite{Elgindy2013} and the barycentric GIM, $\mathbf{P}_B^{(1)}$ using $n = 20, 60, 80, 120; 140$, and $\alpha = -0.25(0.25)2$.}
\label{Comp1}
\end{figure}

To analyze the errors of the barycentric and basis quadratures, we have conducted several numerical experiments on the three test functions ${f_1}(x) = {x^{20}},\;{f_2}(x) = {e^{ - {x^2}}}$, and ${f_3}(x) = 1/(1 + 25\,{x^2})$, which were studied by \cite{Elgindy2013}. The absolute errors (AEs) obtained for $\{f_i\}_{i=1}^3$ are shown in Figures \ref{fig:test1}-\ref{fig:test3}, where one can clearly verify that both quadratures share the same order of error for all $\{f_i\}_{i=1}^3$ with almost matched error values for the third test function, $f_3$.

\begin{figure}[ht]
\centering
\subfigure[]{
\includegraphics[scale=0.3]{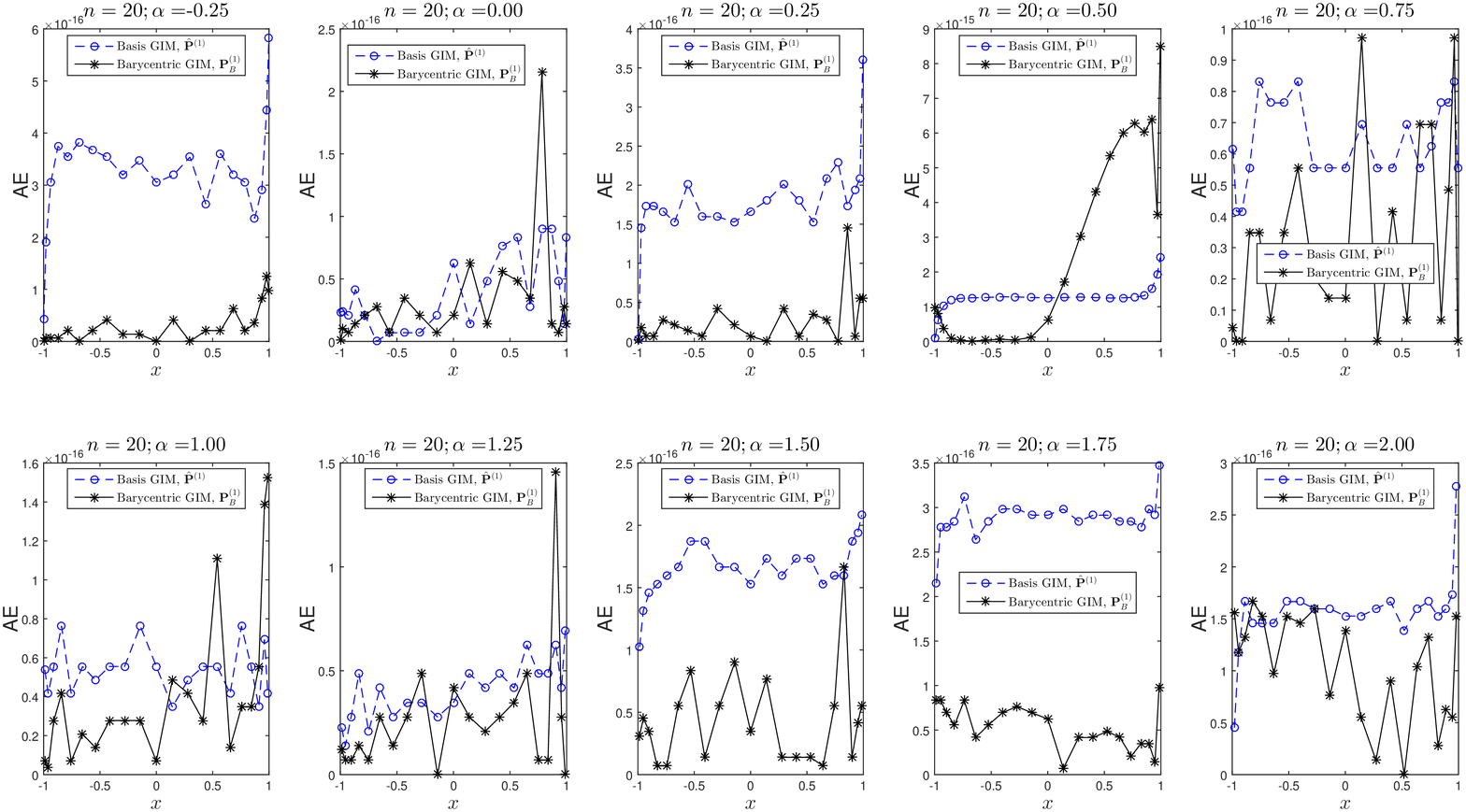}
\label{fig:f120}
}
\subfigure[]{
\includegraphics[scale=0.3]{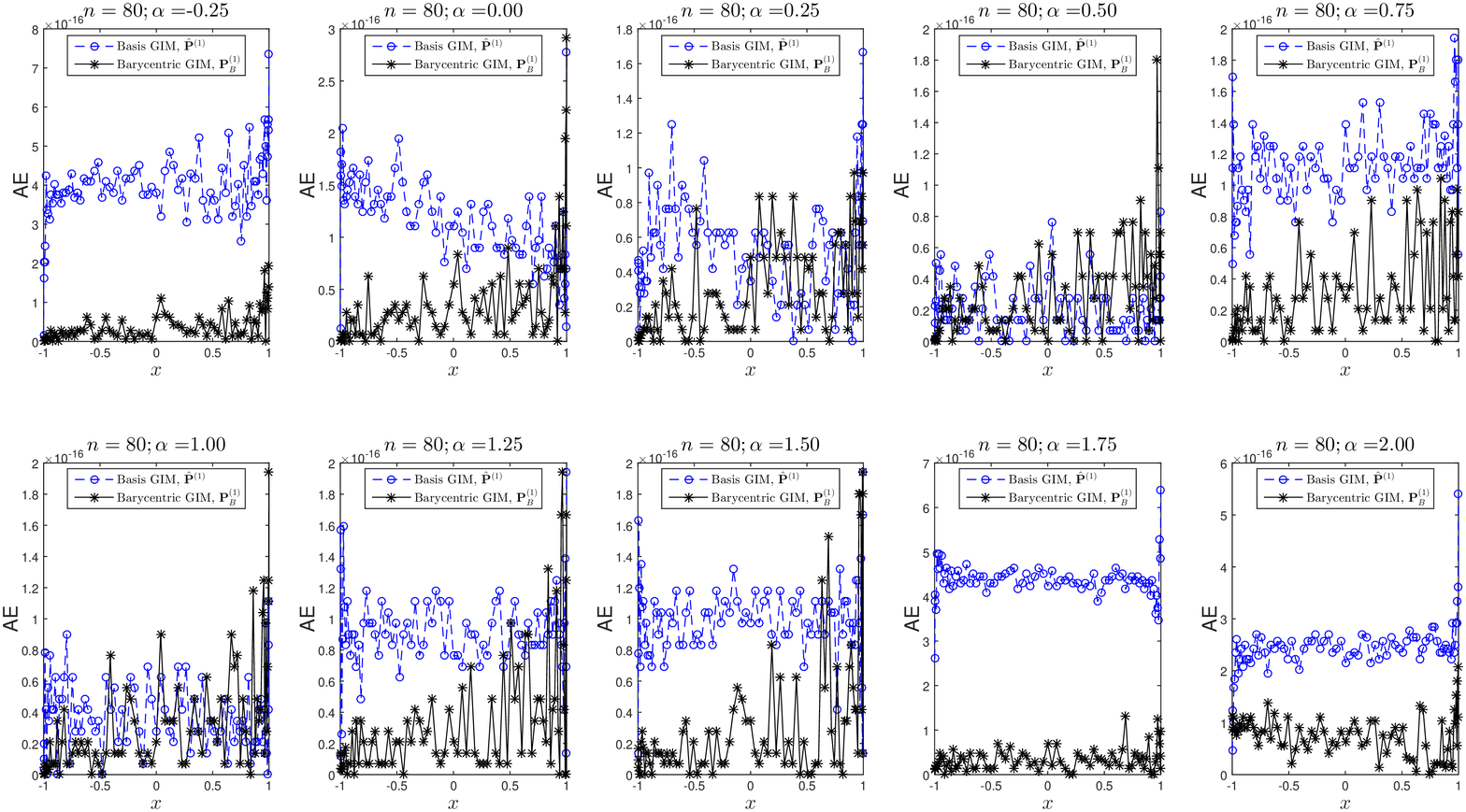}
\label{fig:f180}
}
\caption{The AEs of the barycentric and basis quadratures for $f_1$ on $[-1, 1]$ for $\alpha = -0.25(0.25)2$. Figures \ref{fig:f120} and \ref{fig:f180} show the results for $n = 20; 80$, respectively.}
\label{fig:test1}
\end{figure}
\begin{figure}[ht]
\centering
\subfigure[]{
\includegraphics[scale=0.3]{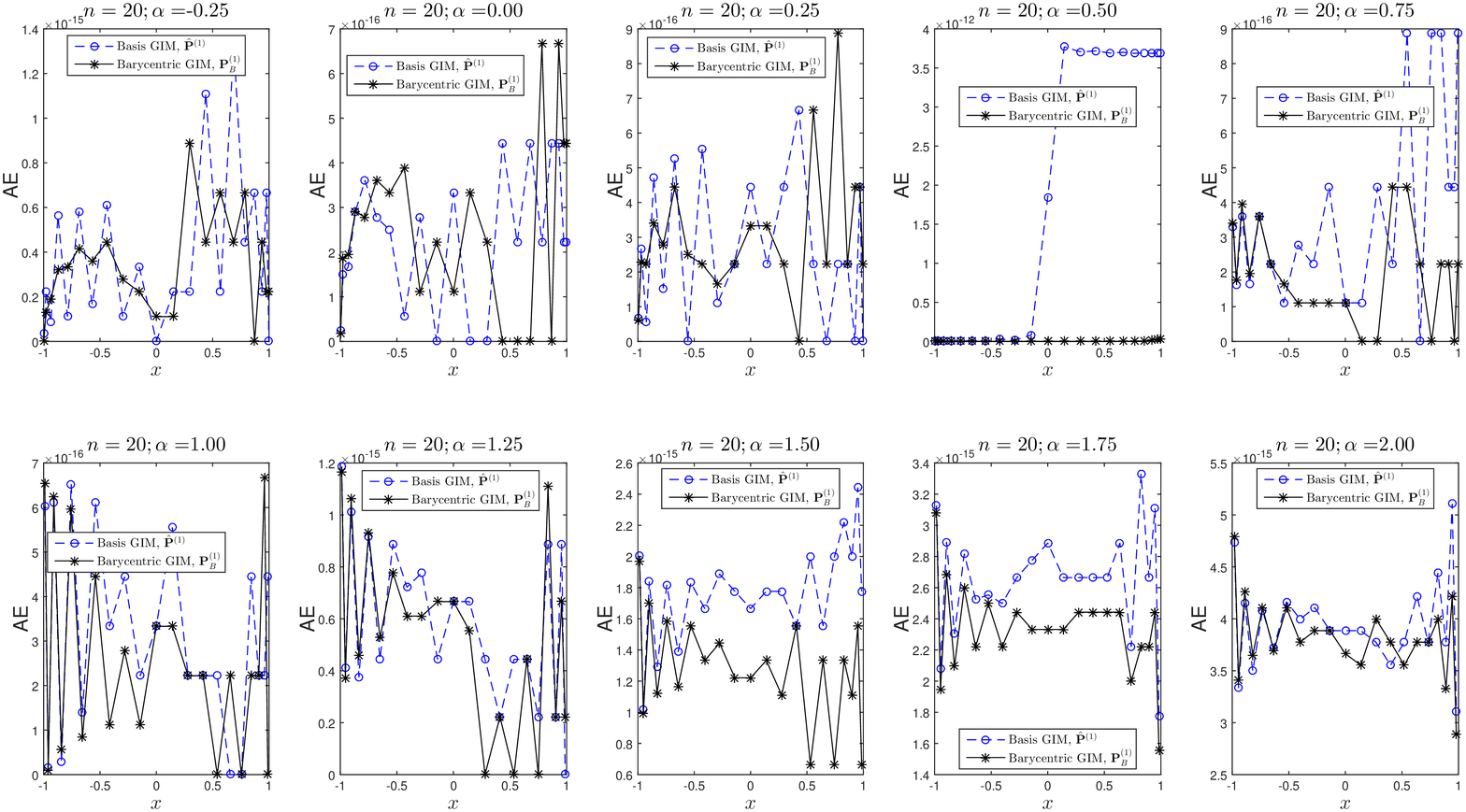}
\label{fig:f220}
}
\subfigure[]{
\includegraphics[scale=0.3]{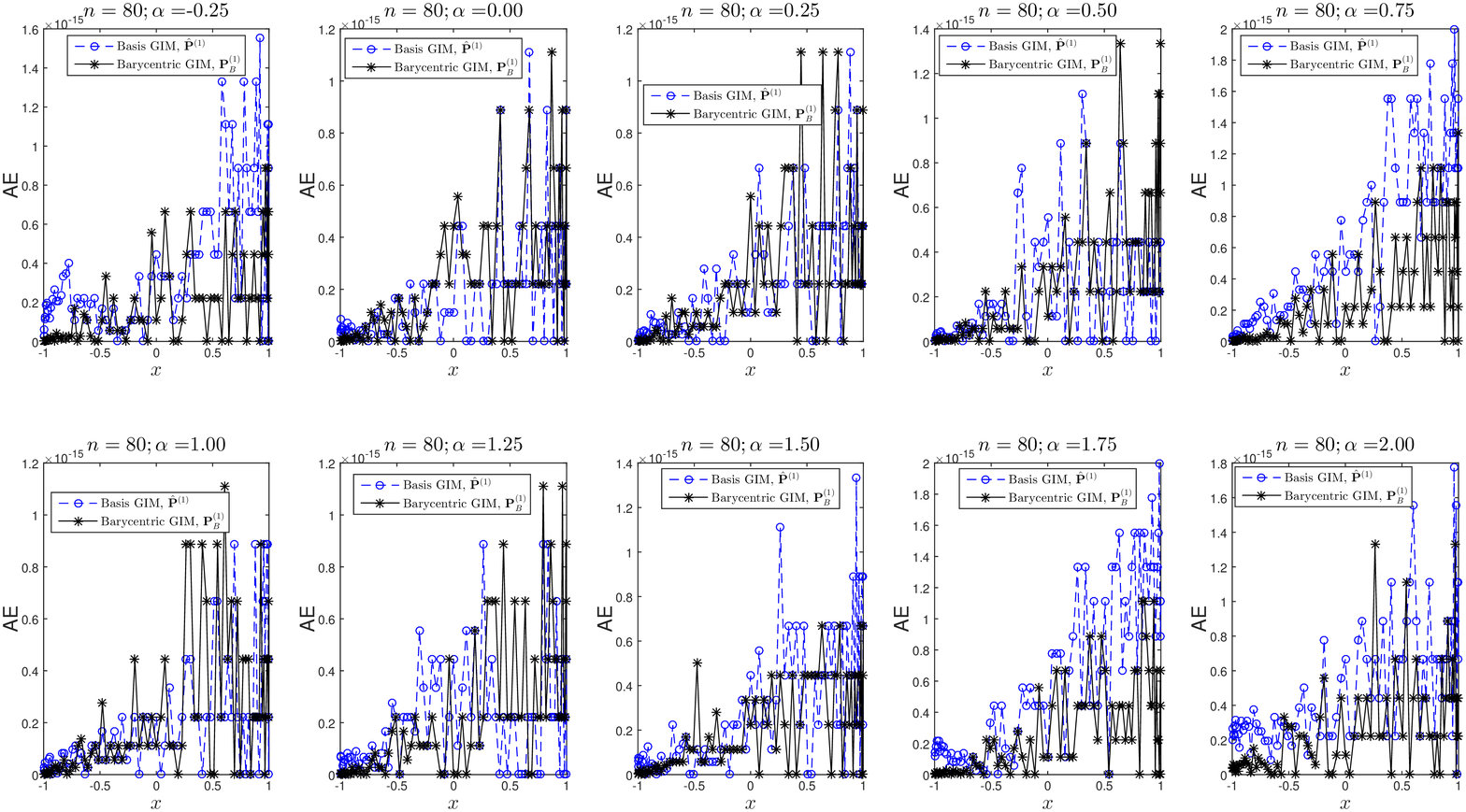}
\label{fig:f280}
}
\caption{The AEs of the barycentric and basis quadratures for $f_2$ on $[-1, 1]$ for $\alpha = -0.25(0.25)2$. Figures \ref{fig:f220} and \ref{fig:f280} show the results for $n = 20; 80$, respectively.}
\label{fig:test2}
\end{figure}
\begin{figure}[ht]
\centering
\subfigure[]{
\includegraphics[scale=0.3]{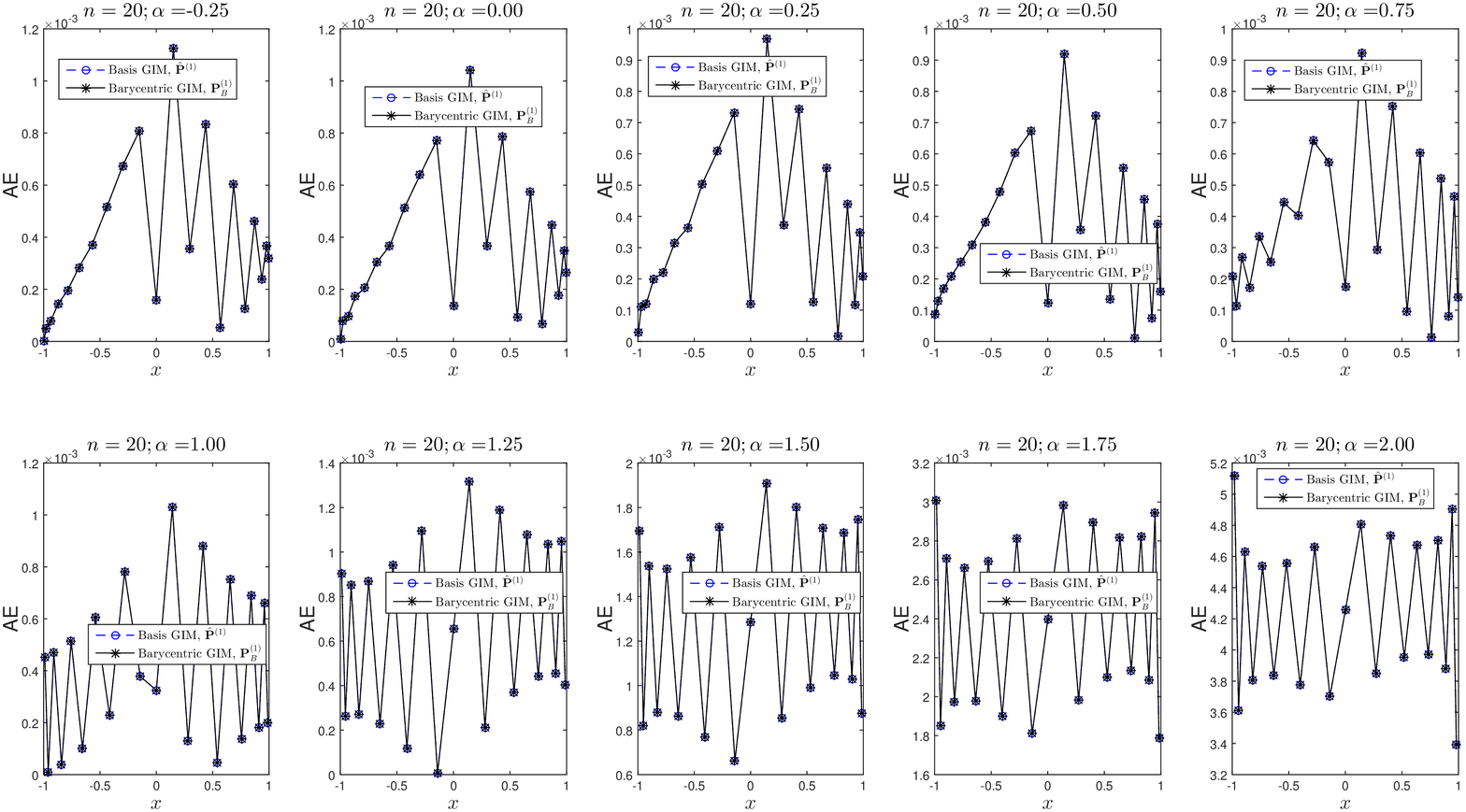}
\label{fig:f320}
}
\subfigure[]{
\includegraphics[scale=0.3]{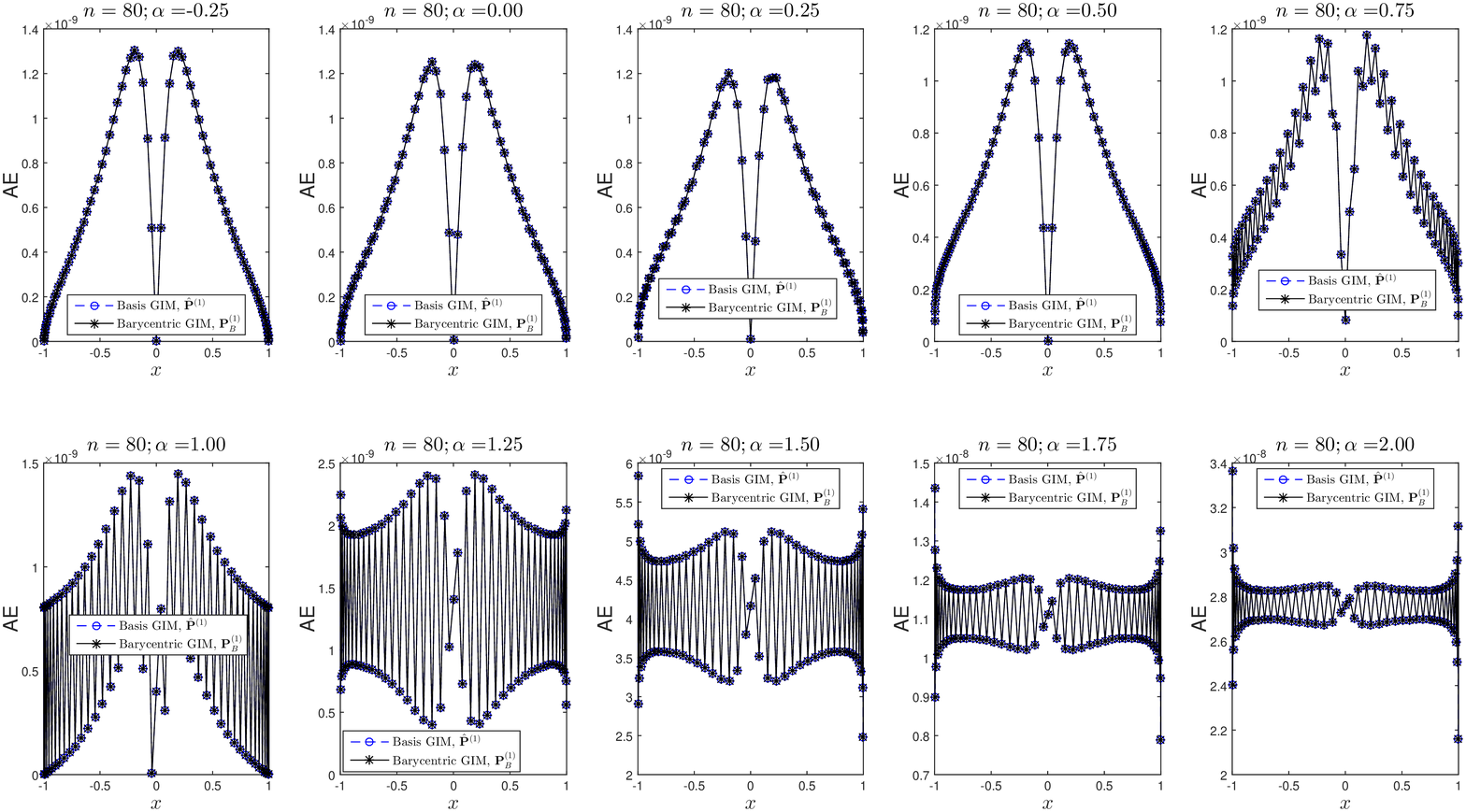}
\label{fig:f380}
}
\caption{The AEs of the barycentric and basis quadratures for $f_3$ on $[-1, 1]$ for $\alpha = -0.25(0.25)2$. Figures \ref{fig:f320} and \ref{fig:f380} show the results for $n = 20; 80$, respectively.}
\label{fig:test3}
\end{figure}

Notice that both the CPU time and AEs were not reported for $n = 40, 100, 160$ as we observed that
\begin{equation}\label{eq:rarecase1}
	\left| {\hat x_{N,k}^{(0.5)} - x_{n,i}^{(\alpha )}} \right| \le \varepsilon_{\text{mach}},
\end{equation}
for certain values of $\alpha$, where 
\begin{equation}
	\hat x_{N,k}^{(0.5)} = \frac{1}{2}\left( {\left( {x_{n,j}^{(\alpha )} + 1} \right){\mkern 1mu} x_{N,k}^{(0.5)} + x_{n,j}^{(\alpha )} - 1} \right)\;\forall k,
\end{equation}
and $\varepsilon_{\text{mach}}$ denotes the machine precision that is approximately equals $2.2204 \times 10^{-16}$ in double precision arithmetic. For instance, we find that Eq. \eqref{eq:rarecase1} is satisfied for $k = 25$ and $i = 33$ using $101$ GG points at $\alpha = 1$, where both $\hat x_{50,25}^{(0.5)}$ and $x_{100,33}^{(1)}$ equals -0.5; thus overflow occurs. In fact, such a rare and unpleasant difficulty could happen whenever,
\begin{equation}\label{eq:rarecase2}
	x_{N,k}^{(0.5)} = \frac{{2x_{n,i}^{(\alpha )} - x_{n,j}^{(\alpha )} + 1}}{{x_{n,j}^{(\alpha )} + 1}},
\end{equation}
for some $i, j \in \{0, \ldots, n\}, k \in \{0, \ldots, N\}$ in exact arithmetic, or
\begin{equation}\label{eq:rarecase22}
	\left| {1 + x_{N,k}^{(0.5)} - \frac{{2\left( {1 + x_{n,i}^{(\alpha )}} \right)}}{{1 + x_{n,j}^{(\alpha )}}}} \right| \le \varepsilon \;\forall i,j;k,
\end{equation}
in finite precision arithmetic, for some relatively small positive number $\varepsilon$. Therefore, a sufficient condition for constructing the barycentric GIM using Algorithm \ref{sec:COTBGFTGSOP1:alg1matrix} is given by
\begin{equation}\label{eq:rarecase3}
	\left| {1 + x_{N,k}^{(0.5)} - \frac{{2\left( {1 + x_{n,i}^{(\alpha )}} \right)}}{{1 + x_{n,j}^{(\alpha )}}}} \right| > \varepsilon \;\forall i,j;k.
\end{equation}
We shall refer to the set,
\begin{equation}
	\mathbb{F}_B = \left\{(n, \alpha):{\text{The Sufficient Condition }}\eqref{eq:rarecase3}\text{ is always satisfied}\right\},
\end{equation}
by the ``barycentric GIM feasible set.''

One approach to construct the barycentric GIM for $(n, \alpha) \not\in \mathbb{F}_B$, is to modify Algorithm \ref{sec:COTBGFTGSOP1:alg1matrix} so that it accomplishes the fundamental property of Lagrange interpolating polynomials,
%\begin{equation}
	%{\mathcal{L}}_{B,n,i}^{(\alpha )}\left( {x_{N,k}^{(0.5)}; - 1,x_{n,j}^{(\alpha )}} \right) = {\delta _{k,i}}\;\forall i,k;j.
%\end{equation}
\begin{equation}\label{eq:fundprop1}
\cal{L}_{B,n,i}^{(\alpha )}\left( {x_{N,k}^{(0.5)}; - 1,x_{n,j}^{(\alpha )}} \right) = 1,\quad {\text{if }}\hat x_{N,k}^{(0.5)} = x_{n,i}^{(\alpha )}\;\forall i,k.
\end{equation}
Algorithm \ref{sec:MCOTBGFTGSOP1:alg1matrix} is a modification to Algorithm \ref{sec:COTBGFTGSOP1:alg1matrix}, which ensures the satisfaction of the Sufficient Condition \eqref{eq:rarecase3}; cf. Appendix \ref{appendix:CAP}. The trick here is to set $\cal{L}_{B,n,i}^{(\alpha )}\left( {x_{N,k}^{(0.5)}; - 1,x_{n,j}^{(\alpha )}} \right) = 1 \;\forall i,j,k$ initially, and then update only the values of $\cal{L}_{B,n,i}^{(\alpha )}\left( {x_{N,k}^{(0.5)}; - 1,x_{n,j}^{(\alpha )}} \right)$ for which Condition \eqref{eq:rarecase3} is satisfied. However, the result of such a modification casts its shadows on the time complexity required for constructing the barycentric GIM, ${\mathbf{P}}_B^{(1)}$. Indeed, Figure \ref{Comp2} shows that the time required for constructing the basis GIM, $\hat {\mathbf{P}}^{(1)}$, derived by \cite{Elgindy2013} becomes shorter than that required for constructing ${\mathbf{P}}_B^{(1)}$ using several values of $n$ and $\alpha$.

\begin{figure}[ht]
\centering
\includegraphics[scale=0.35]{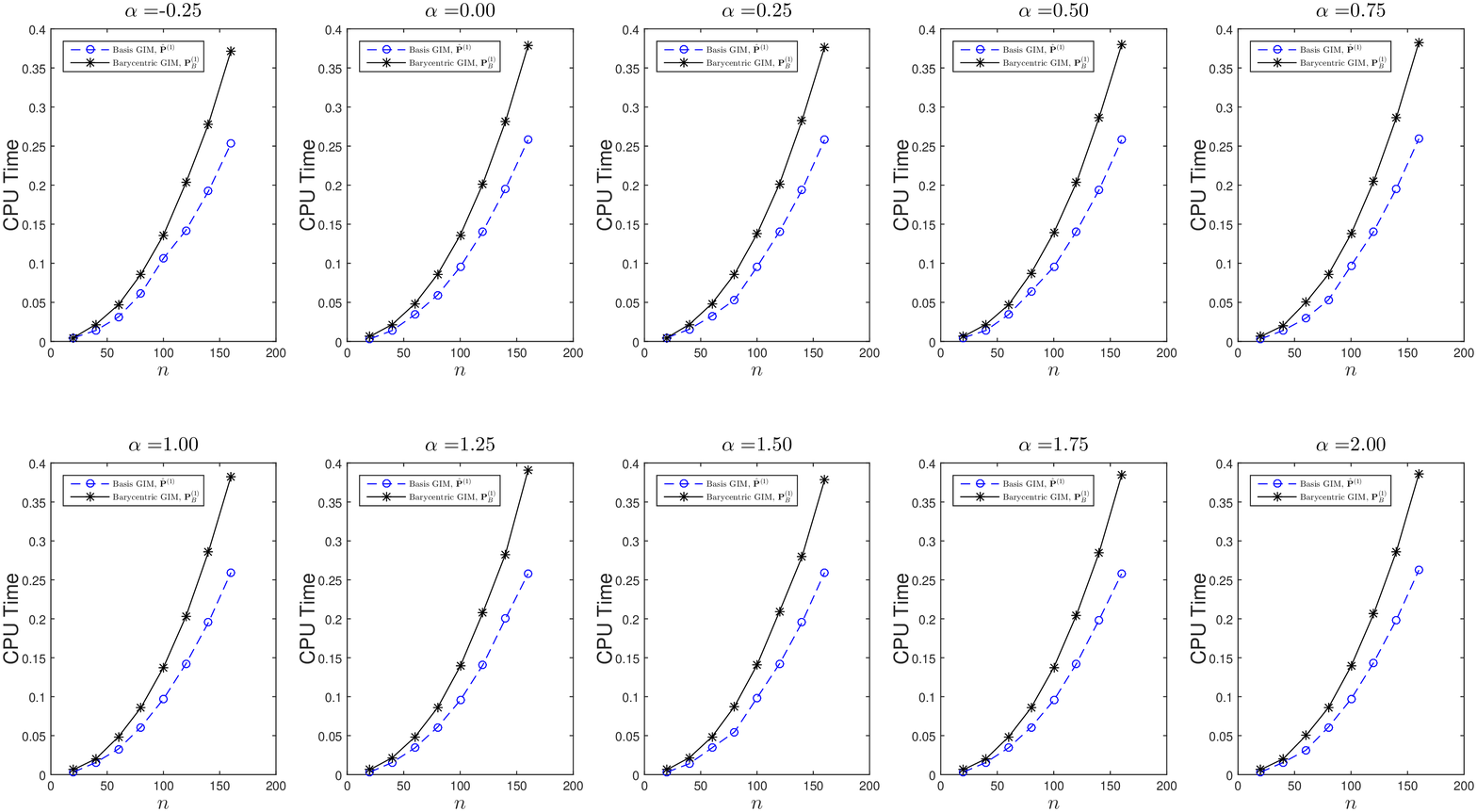}
\caption{The average elapsed CPU time in 10 runs required for the construction of the basis GIM, $\hat {\mathbf{P}}^{(1)}$, derived by \cite{Elgindy2013} and the barycentric GIM, $\mathbf{P}_B^{(1)}$, using $n = 20(20)160$, and $\alpha = -0.25(0.25)2$.}
\label{Comp2}
\end{figure}

Another approach to successfully construct the barycentric GIM for $(n, \alpha) \not\in \mathbb{F}_B$ without applying Property \eqref{eq:fundprop1} is to increase the value of $N$ if Condition \eqref{eq:rarecase22} occurs before applying Algorithm \ref{sec:COTBGFTGSOP1:alg1matrix}; that is, we use instead an $\left\lceil {(n + 3)/2} \right\rceil $-point LG quadrature for calculating the integrals
\begin{equation}\label{eq:intk1}
	\int_{ - 1}^1 {{\cal L}_{B,n,i}^{(\alpha )}\left( {t; - 1,x_{n,j}^{(\alpha )}} \right)\,dt},\quad i,j = 0, \ldots, n.
\end{equation}
 For instance, replacing $N$ with $(N + 1)$ would change the values of $\{x_{N,k}^{(0.5)}\}_{k = 0}^N$ with the possibility of fulfilling Condition \eqref{eq:rarecase3} while exactly calculating the integrals \eqref{eq:intk1} and retaining relatively lower computational cost. Algorithm \ref{sec:TTSCFTCOTBGFTGSOP:alg1matrix} checks for the satisfaction of the Sufficient Condition \eqref{eq:rarecase3}; cf. Appendix \ref{appendix:CAP}. Now running Algorithm \ref{sec:COTBGFTGSOP1:alg1matrix} with the replacement of the statement $N \leftarrow \left\lceil {(n - 1)/2} \right\rceil$ by $N \leftarrow \left\lceil {(n + 1)/2} \right\rceil$, retrieves the previous rapid construction of the barycentric GIM, ${\mathbf{P}}_B^{(1)}$, as verified by Figure \ref{Comp3}.

\begin{rem}
We checked the Sufficient Condition \eqref{eq:rarecase3} using Algorithm \ref{sec:TTSCFTCOTBGFTGSOP:alg1matrix} running on a Windows 10 64-bit operating system endowed with MATLAB V. R2014b (8.4.0.150421) in double precision arithmetic for $n = 1(1)100$, $\alpha = -0.4(0.001)2$, and $\varepsilon = \varepsilon_{\text{mach}}$. Failure to construct the barycentric GIM was only reported at $\alpha = 1$ for $n = 4(12)100$. Therefore, Gegenbauer collocation schemes can be carried out safely and efficiently using any of the aforementioned valid input data.
\end{rem}

\begin{figure}[ht]
\centering
\includegraphics[scale=0.35]{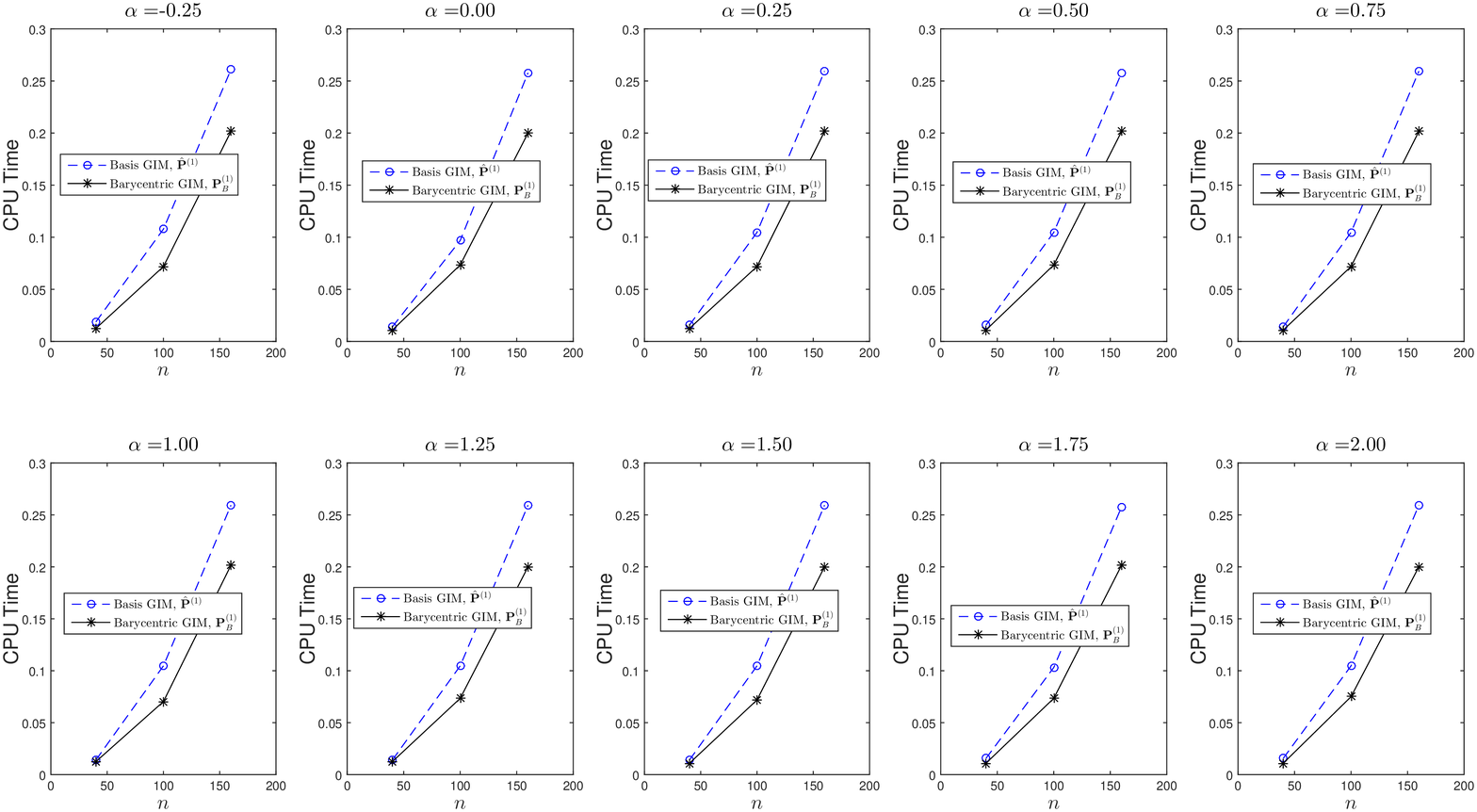}
\caption{The average elapsed CPU time in 10 runs required for the construction of the basis GIM, $\hat {\mathbf{P}}^{(1)}$, derived by \cite{Elgindy2013} and the barycentric GIM, $\mathbf{P}_B^{(1)}$ using $n = 40, 100; 160$, and $\alpha = -0.25(0.25)2$.}
\label{Comp3}
\end{figure}
\subsection{Resolving boundary-value problems using the barycentric Gegenbauer quadrature}
\label{subsec: RBVPUTBG1}
To solve the integral reformulations of differential problems provided with boundary conditions using Gauss collocation methods, one needs to approximate the integral of the unknown solution $y(x)$ on $[-1, 1]$ using Gauss collocation points $\{x_{n,i}^{(\alpha )}\}_{i=0}^n \subset (-1, 1)$; that is the following integral is often required,
\begin{equation}
	I = \int_{ - 1}^1 {y(x)\,dx}.
\end{equation}
Since the barycentric GIM is in principal designed for the GG points, we cannot directly apply Algorithm \ref{sec:COTBGFTGSOP1:alg1matrix} for evaluating $I$. Notice here that we cannot approximate $I$ using a LG quadrature unless the collocation points are the LG points. Therefore, to consider the more general case, we shall modify Algorithm \ref{sec:COTBGFTGSOP1:alg1matrix} to work for general GG points. At first, denote the point $1$ by ${x_{n+1,j}^{(\alpha )}}$. Then Eqs. \eqref{eq:pbGIMe1} can be written as follows:
\begin{equation}\label{eq:k2862015iss1}
	p_{B,n + 1,i}^{(1)} = \sum\limits_{k = 0}^N {\;\varpi _{N,k}^{(0.5)}{\mkern 1mu} {\mathcal{L}}_{B,n,i}^{(\alpha )}\left( {x_{N,k}^{(0.5)}; - 1,x_{n,n + 1}^{(\alpha )}} \right)} ,\quad i = 0, \ldots ,n,
\end{equation}
where $p_{B,n + 1,i}^{(1)}, i = 0, \ldots, n$, are the elements of the additional row of the barycentric GIM, ${\mathbf{P}}_{B,n+1}^{(1)}$, corresponding to the point $1$. Since $\hat x_{N,k}^{(0.5)} = x_{N,k}^{(0.5)}\; \forall k$ in this case, the Sufficient Condition \eqref{eq:rarecase3} is now simplified to
\begin{equation}
	\left| {x_{N,k}^{(0.5)} - x_{n,i}^{(\alpha )}} \right| > \varepsilon \;\forall i;k.
\end{equation}
Since $n > N \; \forall n \ge 1$, both LG and GG points share the zero value if both $n$ and $N$ are even. For instance, for $n = 4$, we find that ${x_{2,1}^{(0.5)} = x_{4,2}^{(\alpha)}} = 0\, \forall \alpha > -1/2$, and again overflow occurs. To overcome this issue we increase the number of LG points as discussed before so that $N$ is replaced with $N + 1$. This convenient technique is adopted in Algorithm \ref{sec:COARBGFTP1:alg1matrix}; cf. Appendix \ref{appendix:CAP}. Notice here that Algorithm \ref{sec:COARBGFTP1:alg1matrix} is implemented assuming ${\mathbf{P}}_B^{(1)}$ for the GG collocation points is not required. If this is not the case, then the barycentric weights are already computed using Algorithm \ref{sec:COTBGFTGSOP1:alg1matrix}, and we can safely remove this partial procedure to gain more efficiency when calculating the row barycentric GIM corresponding to the point $1$. This is depicted in Algorithm \ref{sec:MCOARBGFTP1:alg1matrix}.

\begin{rem}
The barycentric GIM can be slightly modified to work for any set of arbitrary points $\{x_k\}_{k = 0}^n$ by following Algorithm \ref{sec:COTBGFAASOP:alg1matrix} in Appendix \ref{appendix:CAP}.
\end{rem}

\subsection{Computational algorithms for the optimal barycentric GIM}
\label{subsec:CAFTOBG1}
Similar to the work of \cite{Elgindy2013}, the computational cost of $\mathbf{P}_{OB}^{(1)}$ can be reduced significantly for arbitrarily symmetric set of points $\{x_k\}_{k=0}^n$ if $m$ is even. Indeed, in this case, $\int_{ - 1}^{{x_k}} {G_{m + 1}^{(\alpha _k^*)}(x)\,dx}  = \int_{ - 1}^{ - {x_k}} {G_{m + 1}^{(\alpha _k^*)}(x)\,dx}\, \forall k$; thus $\left\{ {\alpha _k^*} \right\}_{k = 0}^{\left\lceil {n/2} \right\rceil  - 1} = \left\{ {\alpha _{n - k}^*} \right\}_{k = 0}^{\left\lceil {n/2} \right\rceil  - 1}$, which implies that $\left\{ {z_{m,k,i}^{(\alpha _k^*)},\xi _{m,k,i}^{(\alpha _k^*)}} \right\}_{i = 0}^m = \left\{ {z_{m,n - k,i}^{(\alpha _{n - k}^*)},\xi _{m,n - k,i}^{(\alpha _{n - k}^*)}} \right\}_{i = 0}^m$, for $k = 0, \ldots, {\left\lceil {n/2} \right\rceil }-1$. Hence, $\left\{ {z_{m,k,i}^{(\alpha _k^*)},\xi _{m,k,i}^{(\alpha _k^*)}} \right\}_{i = 0}^m$ can be stored for the first $({\left\lceil {n/2} \right\rceil } - 1)$ iterations, and invoked later in the next iterations. Algorithms \ref{sec1:alg:1} and \ref{sec1:alg:2} are two efficient algorithms for the construction of the optimal barycentric GIM for any non-symmetric/symmetric set of integration points, respectively; cf. Appendix \ref{appendix:CAP}. For a large number of expansion terms, Chebyshev and Legendre quadratures often behave optimally as specified by the used error norm; cf. \cite{Elgindy2013}. Therefore, both algorithms provide the user with the flexibility to choose two parameter inputs $m_{\max}$ and $\alpha_a$ at which the algorithms construct the Chebyshev/Legendre quadratures instead. Moreover, the parameter input $\alpha_b$ adds further stability to the algorithms in the occasions, where $\alpha_k^*$ lies in the critical interval $(-0.5, -0.5 + \varepsilon)$. The parameter input $r$ is ideally chosen from the interval $[1, 2]$ to hamper the extrapolatory effect of the optimal Gegenbauer quadrature caused by the narrowing behavior of the Gegenbauer weight function for increasing values of $\alpha$; cf. \cite{Elgindy2013}. For a non-symmetric set of integration nodes $\{x_k\}_{k=0}^n$ with $m \le m_{\max}: m$ is even and $1 \in \{x_k\}_{k=0}^n$, $M$ should be replaced with $(M + 1)$ in Algorithm \ref{sec1:alg:1} if $M$ is even, as we discussed earlier. This procedure should also be carried out in Algorithm \ref{sec1:alg:2} for symmetric sets of integration nodes with $m \le m_{\max}$ and $1 \in \{x_k\}_{k=0}^n$, since $m$ is always an even integer in such cases.

Since Algorithms \ref{sec1:alg:1} and \ref{sec1:alg:2} work for any arbitrary set of nodes $\{x_k\}_{k=0}^n$, the Sufficient Condition \eqref{eq:rarecase3} for $m > m_{\max}$ now becomes
\begin{equation}\label{eq:rarecase3OB}
	\left| {x_{M,s}^{(0.5)} - \frac{{1 - {x_k} + 2x_{m,i}^{({\alpha _a})}}}{{1 + {x_k}}}} \right| > \varepsilon \;\forall i,s;k,
\end{equation}
which can be checked using Algorithm \ref{sec:TTSKIMO1:alg1matrix}. We refer to the set,
\begin{equation}
	\mathbb{F}_{OB,1}^{(m_{\max})} = \left\{(m, \alpha_a):m > m_{\max}; {\text{the Sufficient Condition }}\eqref{eq:rarecase3OB}\text{ is always satisfied}\right\},
\end{equation}
by the ``optimal barycentric GIM feasible set for $m > m_{\max}$.'' It is important here to mention that in Gegenbauer collocation schemes, the dynamics is often enforced at the GG points; cf. \cite{Elgindy2013,Elgindy2013a,Elgindy2016,Elgindy2012c,Elgindy2012d}. Therefore, the input set of integration nodes $\{ {x_k}\} _{k = 0}^n$ in Algorithm \ref{sec1:alg:2} is frequently taken as the set of GG points $\{x _{n,k}^{(\alpha )}\}_{k = 0}^n$. Hence, the sufficient condition for constructing $\mathbf{P}_{OB}^{(1)}$ for $m \le m_{\max}$ reads
\begin{equation}\label{eq:rarecase3OBmlmmax}
	\left| {x_{M,s}^{(0.5)} - \frac{{1 - x_{n,k}^{(\alpha )} + 2\,z_{m,k,i}^{(\alpha _k^*)}}}{{1 + x_{n,k}^{(\alpha )}}}} \right| > \varepsilon \;\forall i,s;k.
\end{equation}
We refer to the set,
\begin{equation}
	\mathbb{F}_{OB,2}^{(m_{\max})} = \left\{(n, m, \alpha):m \le m_{\max}; {\text{the Sufficient Condition }}\eqref{eq:rarecase3OBmlmmax}\text{ is always satisfied}\right\},
\end{equation}
by the ``optimal barycentric GIM feasible set for $m \le m_{\max}$.'' The above argument implies that the set,
%\begin{equation}
	%\mathbb{F}_{OB}^{(m_{\max})} = \left\{(n, m, \alpha, \alpha_a): \text{the Sufficient Conditions }\eqref{eq:rarecase3OB}\text{ and }\eqref{eq:rarecase3OBmlmmax}\text{ are always satisfied}\right\},
%\end{equation}
\begin{equation}
	\mathbb{F}_{OB}^{({m_{\max }})} = \left\{ \begin{array}{l}
\mathbb{F}_{OB,1}^{({m_{\max }})},\quad m > {m_{\max }},\\
\mathbb{F}_{OB,2}^{({m_{\max }})},\quad m \le {m_{\max }},
\end{array} \right.
\end{equation}
is the optimal barycentric GIM feasible set.
\begin{rem}
The phrase `If $M = N$ then set $\mathbf{P} = \hat {\mathbf{P}}$ with $\alpha = 0.5$;' in \cite[Algorithms 2.1 \& 2.2]{Elgindy2013} should be carried out with each $i$th-indexed GG point $x_{i}$ replaced with the corresponding arbitrary point while keeping the $j$th-indexed GG points $x_{j}$ the same. This should be straightforward and the implementation should follow that of Algorithm \ref{sec:COTBGFAASOP:alg1matrix} in Appendix \ref{appendix:CAP}. We have also noticed a typo in \cite[Algorithm 2.2]{Elgindy2013}, where the phrase `$N$ is even' in the input should be replaced with `$M$ is even.' In turns, the condition `$i \le N/2$' in Step 4 should be correctly replaced with `$i \le \left\lfloor {N/2} \right\rfloor$' to cover both cases when $N$ is even or odd.
\end{rem}

\begin{rem}
Since most of the current state of the art software such as MATLAB are optimized for operations involving matrices and vectors, all of the proposed algorithms are vectorized to run much faster than the corresponding codes containing loops.
\end{rem}

\section{Numerical examples}
\label{sec: NEAA1}
In this section, we apply the developed barycentric GIMs and quadratures on three well-studied test examples with known exact solutions in the literature. Comparisons with other competitive numerical schemes are presented to assess the accuracy and efficiency of the current work. The numerical experiments were conducted on a personal laptop equipped with an Intel(R) Core(TM) i7-2670QM CPU with 2.20GHz speed running on a Windows 10 64-bit operating system.

\paragraph{\textbf{Example 1}} Consider the following Fredholm integro-differential equation:
\begin{equation}\label{sec3:eq:test7}
  y'(x) - y(x) - \int_0^1 {{e^{s  x}}\; y(s)\;  ds}  = \frac{{1 - {e^{x + 1}}}}{{x + 1}},\quad  y(0) = 1,
	\end{equation}
with the exact solution $y(x) = e^x$. This problem was previously solved by \cite{Elgindy2013a} using a hybrid Gegenbauer integration method (HGIM). Following the numerical scheme developed by \cite{Elgindy2013a} together with the obtained barycentric GIMs results in the following algebraic system of linear equations:
\begin{equation}\label{sec3:eq:test7ls}
{w_j} - \sum\limits_{i = 0}^n {\left( {p_{B,j,i}^{(1)} + \sum\limits_{k = 0}^n {p_{B,n + 1,i}^{(1)}p_{B,j,k}^{(1)}{e^{x_{n,k}^{(\alpha )}x_{n,i}^{(\alpha )}}}} } \right){\mkern 1mu} {w_i} - \sum\limits_{i = 0}^m {p_{OB,j,i}^{(1)}\;{r_{j,i}}}  - 1 = 0} ,\quad j = 0, \ldots ,n,
\end{equation}
where ${w_j} \approx y\left(x_{n,j}^{(\alpha )}\right)\,\forall j; r(x) = \left(1 - {e^{x + 1}}\right)/(x + 1).$ We refer to the present method by the hybrid barycentric Gegenbauer integration method (HBGIM). We implemented the developed algorithms for the set of feasible $3$-tuples %$\{(10,\alpha)\}_{\alpha = -0.4(0.1)}^1 \subset \mathbb{F}_B$ and
$\{(10,14,\alpha)\}_{\alpha = -0.4(0.1)}^1$ $\subset \mathbb{F}_{OB}^{(20)}$.
%$\{(10,14,\alpha,0)\}_{\alpha = -0.4(0.1)}^1$ $\subset \mathbb{F}_{OB}^{(20)}$.
%with $(14, 0) \in \mathbb{F}_{OB,1}^{(m_{\max})}$
%$m = 14$ and $m_{\max} = 20$.
The resulting algebraic linear system of equations were solved using MATLAB ``mldivide'' Algorithm provided with MATLAB V. R2014b (8.4.0.150421). Figure \ref{FEx1} shows the maximum absolute errors (MAEs) of the present method. As can be observed from the results, the $\max_{\alpha = -0.4:0.1:1} \text{MAE}$ of the present method is about $9.948 \times 10^{-14}$ obtained at $\alpha = 1$ versus approximately $4.201 \times 10^{-13}$ for the HGIM obtained at $\alpha = -0.4$. The best MAE $\approx 5.329 \times 10^{-15}$ was obtained at $\alpha = 0.7$ in $0.017$ seconds. The reported $2$-norm condition number, ${\kappa _2}$, of the linear system is approximately bounded by $35.27 \le {\kappa _2} \le 42.87 \;\forall \alpha$.

\begin{figure}[ht]
\centering
\includegraphics[scale=0.7]{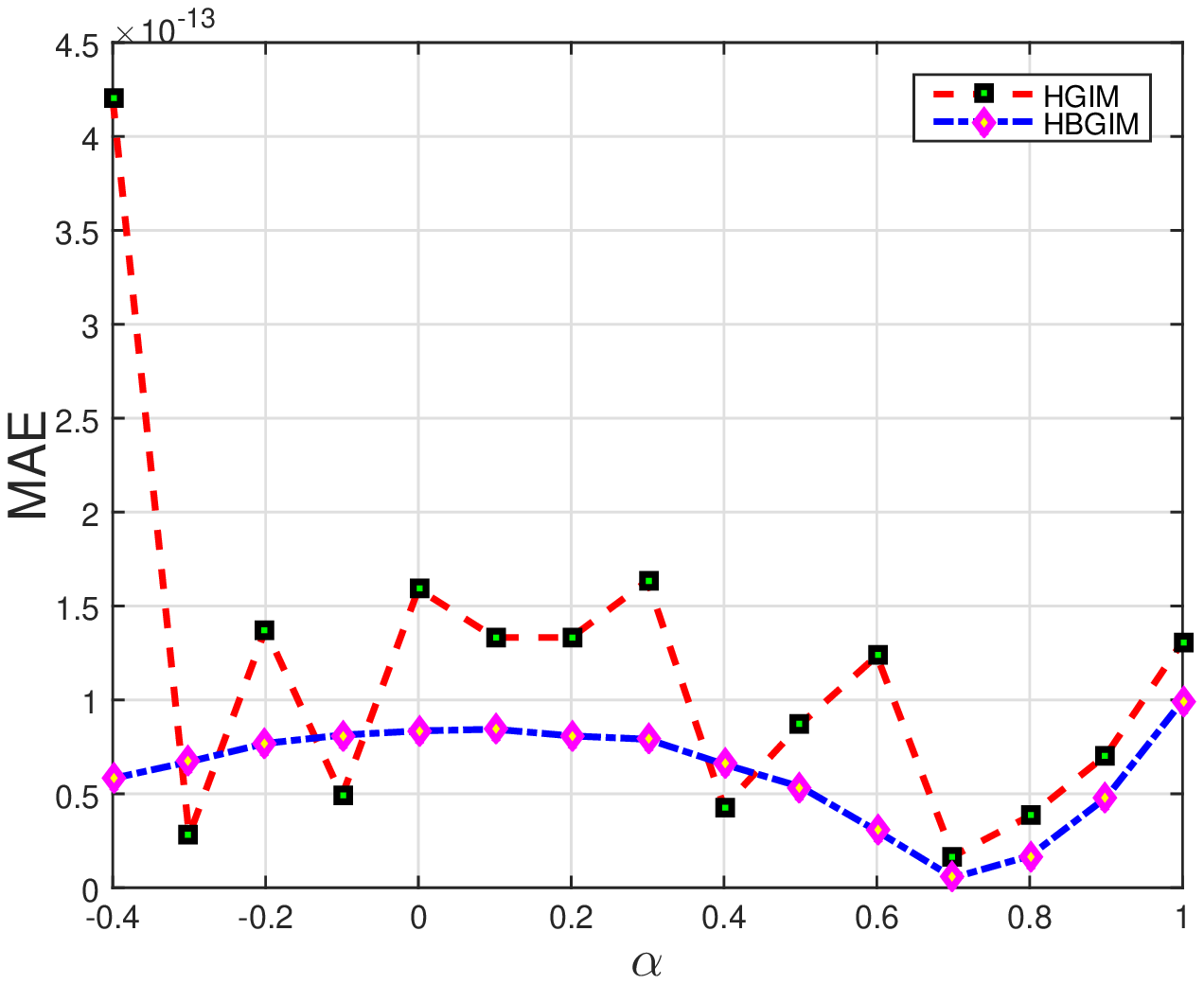}
\caption{The figure shows the MAEs of the HGIM versus the HBGIM for the set of feasible $3$-tuples $\{(10,14,\alpha)\}_{\alpha = -0.4(0.1)}^1$ $\subset \mathbb{F}_{OB}^{(20)}$.}
%$4$-tuples $\{(10,14,\alpha,0)\}_{\alpha = -0.4(0.1)}^1$ $\subset \mathbb{F}_{OB}^{(20)}$.}
\label{FEx1}
\end{figure}

\paragraph{\textbf{Example 2}} Consider the following nonlinear boundary value problem:
\begin{equation}\label{sec:eq:v1}
- {c_\alpha }\alpha \left( {\int_0^1 {u(t)\,dt} } \right)\,u''(x) + {u^5}(x) = 0,\;\;0 < x < 1,
	\end{equation}
	such that $u(0) = 1, u(1) = \sqrt{2}/2, {c_\alpha } = 4/\left( {3\,\alpha \left( {2\sqrt 2  - 2} \right)} \right);\alpha (q) = 1/q\,\forall q > 0.$ The exact solution is $u(x) = 1/\sqrt{1 + x}$ \cite{Themistoclakis2015}. Notice here that the coefficient of the second derivative of the unknown solution $u$ depends upon the integral of $u$ itself, which in turn depends on the whole solution domain $[0, 1]$ rather than on a single point. Therefore, the boundary value problem is classified as a ``nonlocal'' nonlinear problem. This problem was studied by \cite{Themistoclakis2015} and was solved using an iterative scheme (IS). The present HBGIM results in the following nonlinear algebraic system of equations:
\begin{equation}\label{sec3:eq:testv1}
P_B^{(2)}{U^{(5)}} - P_{B,n + 1}^{(2)}{U^{(5)}}X + \left[ {4\left( {4 - 3\sqrt 2 } \right)X - 8\left( {\sqrt 2  - 1} \right)\left( {U - 1} \right)} \right]\oslash\left( {3P_{B,n + 1}^{(1)}U} \right),
\end{equation}
where $U = {\left[ {{U_0},{U_1}, \ldots ,{U_n}} \right]^T} \approx {\left[ {u\left( {x_{n,0}^{(\alpha )}} \right),u\left( {x_{n,1}^{(\alpha )}} \right), \ldots ,u\left( {x_{n,n}^{(\alpha )}} \right)} \right]^T}, X = {\left[ {x_{n,0}^{(\alpha )},x_{n,1}^{(\alpha )}, \ldots ,x_{n,n}^{(\alpha )}} \right]^T};$
\[{U^{(5)}} = \underbrace {U \circ U \ldots  \circ U}_{5 - {\text{times}}},\]
$\circ$ and $\oslash$ denote the Hadamard product and division, respectively. We implemented the developed algorithms for the set of feasible pairs $\{ (n,\alpha ):n = 6,7,9;\alpha  =  - 0.4(0.1)1\}  \subset \mathbb{F}_B$. The nonlinear system \eqref{sec3:eq:testv1} was solved using MATLAB ``fsolve'' solver with ``TolX'' set at $\varepsilon_{\text{mach}}$. Figure \ref{fig:TEX2} shows the MAEs of the present method while Figure \ref{fig:TEX22} shows the number of correct digits ${\text{c}}{{\text{d}}_n}: =  - {\log _{10}}\left[ {{{\max }_{0 \le i \le n}}\left| {u({x_i}) - {U_i}} \right|} \right]$ obtained in each case. Clearly, the present numerical scheme achieves a very rapid convergence rate using relatively small number of barycentric quadratures terms. For instance, the IS of \cite{Themistoclakis2015} requires $257$ points to obtain $6$ correct digits versus only $10$ GG points to achieve a larger number of correct digits for the present method with an elapsed time of about $0.01$ seconds for all experimental values of $\alpha$. Both numerical tests confirm the efficiency of the proposed numerical schemes.

\begin{figure}[ht]
\centering
\includegraphics[width=13cm,height=8cm]{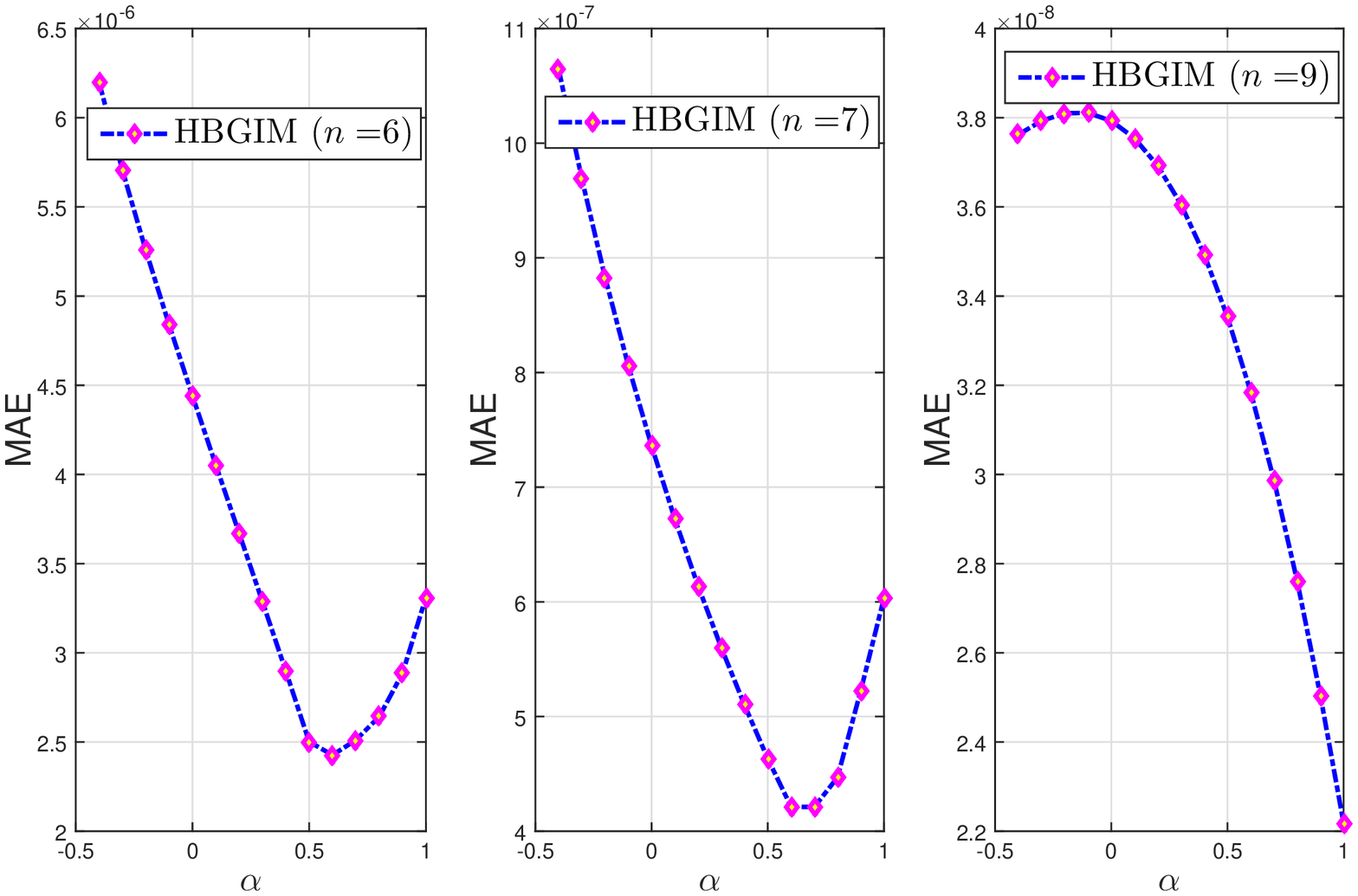}
\caption{The MAEs of the HBGIM for $\{ (n,\alpha ):n = 6,7,9;\alpha  =  - 0.4(0.1)1\}  \subset \mathbb{F}_B$.}
\label{fig:TEX2}
\end{figure}

\begin{figure}[ht]
\centering
\includegraphics[width=16cm,height=8cm]{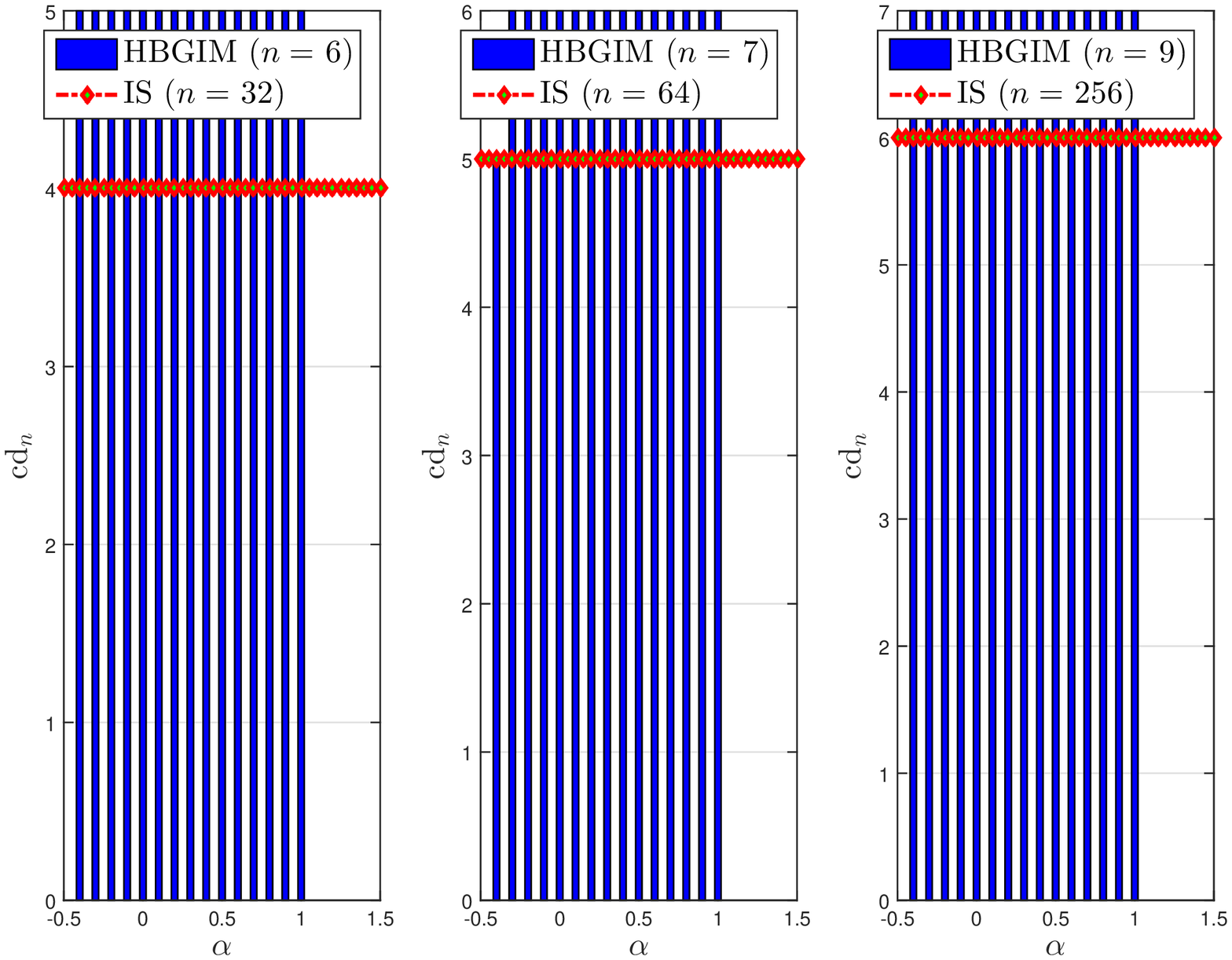}
\caption{The cd$_n$ for the HBGIM and the IS of \cite{Themistoclakis2015}.}
\label{fig:TEX22}
\end{figure}
\paragraph{\textbf{Example 3}} Consider the following second-order one-dimensional hyperbolic telegraph equation,
\begin{equation}
	{u_{tt}}(x,t) + 2{\pi ^2}{u_t}(x,t) + {\pi ^2}u(x,t) = {u_{xx}}(x,t) + {e^{ - t}}\sin (\pi x),\quad 0 \le x \le 1, t > 0,
\end{equation}
provided with the initial conditions,
\begin{align}
	u(x,0) &= \sin (\pi x),\\
	{u_t}(x,0) &= - \sin (\pi x),
\end{align}
and the following Dirichlet boundary conditions,
\begin{align}
	u(0,t) = 0;\\
	u(1,t) = 0.
\end{align}
The exact solution of the above problem is $u(x,t) = {e^{ - t}}\sin (\pi x)$ \cite{Luo2013}. We solved the problem using the numerical scheme developed by \cite{Elgindy2016} together with the obtained barycentric GIMs. The developed algorithms were carried out using the feasible $3$-tuple $(8,8,0) \in \mathbb{F}_{OB}^{(20)}$.
%$4$-tuple $\{(8,8,0,0)\} \subset \mathbb{F}_{OB}^{(20)}$.
%pair $(8,0) \in \mathbb{F}_B$ with %$(8, 0) \in \mathbb{F}_{OB,1}^{(m_{\max})}$
%$m = 8$ and $m_{\max} = 20$.
The plots of the exact solution, its bivariate shifted Gegenbauer interpolant ${P_{n,n}}u(x,t)$ (see \cite{Elgindy2016}), and the absolute error function
\begin{equation}
	{E}_{n,n}(x,t) = \left| u(x,t) - {P_{n,n}}u(x,t) \right|, \quad (x,t) \in D_{1,1}^2,
\end{equation}
%using $n = m = 8$,
are shown in Figure \ref{Ex3err}, where $D_{1,1}^2 = [0,1] \times [0,1]$. A comparison with \cite{Luo2013}'s fourth-order method based on cubic Hermite interpolation \cite{Luo2013} and the present method is also shown in Table \ref{sec:numerical:tab:Ex3}. The plots and the numerical comparisons show the power of the present method as proven in the recognized rapid convergence rates and the produced errors with very small magnitudes using relatively small number of expansion terms. For instance, \cite{Luo2013}'s fourth-order method \cite{Luo2013} yields a MAE of order $10^{-08}$ using $49 \times 49$ collocation points in both directions. Conversely, a MAE of order $10^{-09}$ is achieved by the present method using only $11 \times 11$ collocation points.
\begin{figure}[ht]
\centering
\includegraphics[scale=0.6]{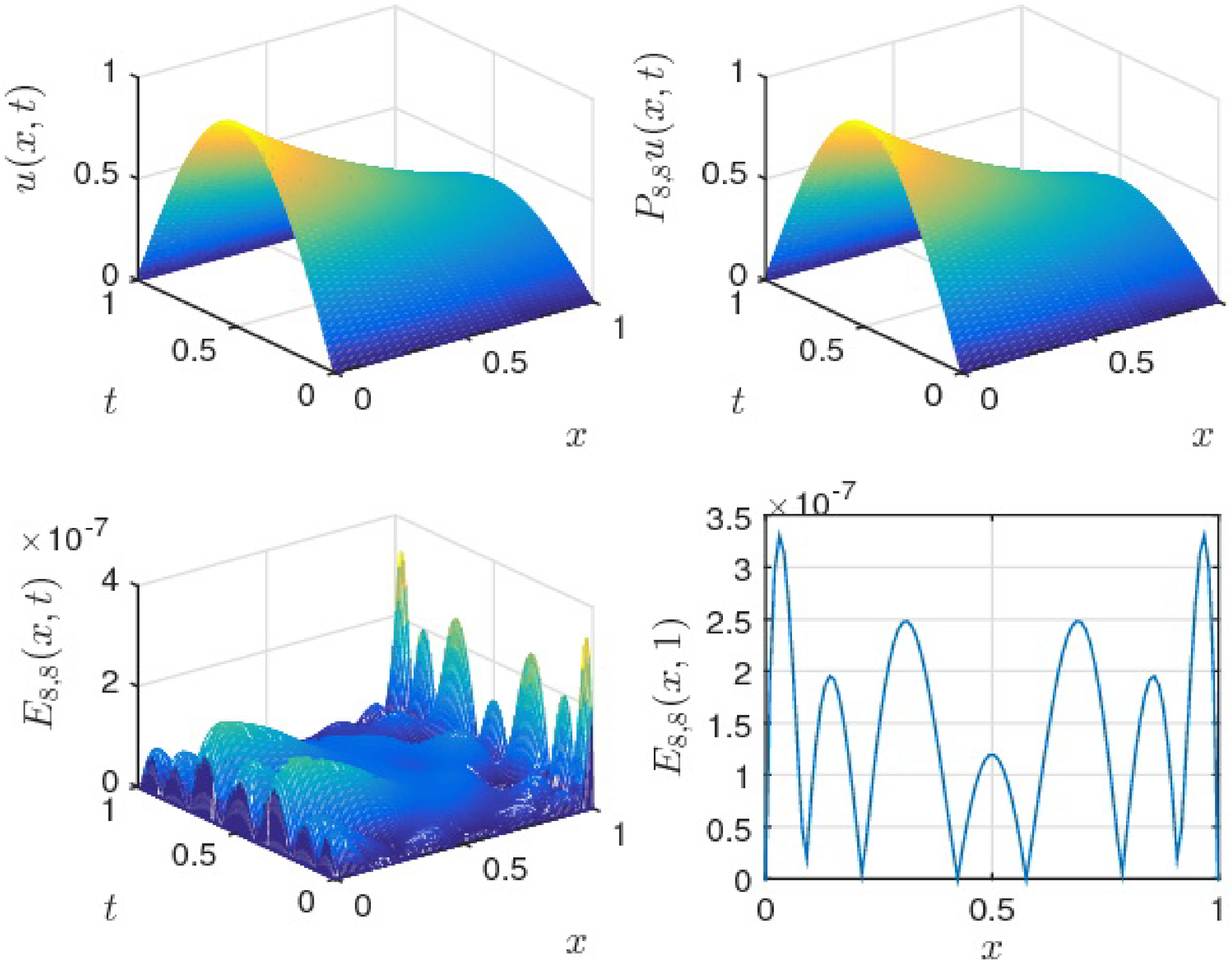}
\caption{The numerical simulation of the present method on Example 3. The figure shows the plots of the exact solution $u(x,t)$  on $D_{1,1}^2$ (upper left), its approximation $P_{8,8}u(x,t)$ (upper right), the absolute error function ${E}_{8,8}(x,t)$ (lower left), and its values at the final time, ${E}_{8,8}(x,1)$ (lower right). The barycentric matrix and its optimal partner are both square of size $9$. The plots were generated using $100$ linearly spaced nodes in the $x$- and $t$-directions from $0$ to $1$.}
\label{Ex3err}
\end{figure}
\begin{table}[ht]
\begin{center} % put inside center environment
\scalebox{0.6}{
\resizebox{\textwidth}{!}{ %
\begin{tabular}{cc}
\toprule
\multicolumn{2}{c}{\textbf{Example 3}} \\
\cmidrule(r){1-2}
\textbf{\cite{Luo2013}'s method \cite{Luo2013}} & \textbf{Present method}\\
$(h)/(MAE); k = 1/48$ & $(n)/(MAE)$\\
\cmidrule(r){1-2}
$(1/12)/(3.702 \times 10^{-06})$ & $(8)/(3.303 \times 10^{-07})$\\
$(1/24)/(2.310 \times 10^{-07})$ & $(10)/(1.596 \times 10^{-09})$\\
$(1/48)/(1.451 \times 10^{-08})$ & $(12)/(5.145 \times 10^{-12})$\\
$(1/96)/(9.952 \times 10^{-10})$ & $(14)/(1.849 \times 10^{-14})$\\
\bottomrule
\end{tabular}
}}
\caption{A comparison of Example 3 between \cite{Luo2013}'s fourth-order method \cite{Luo2013} and the current method. The table lists the MAEs at $t = 1$. The results of \cite{Luo2013}'s method \cite{Luo2013} are quoted from Ref. \cite{Luo2013}.}
\label{sec:numerical:tab:Ex3}
\end{center}
\end{table}
\section{Conclusion and discussion}
\label{conc}
Novel GIMs and quadratures are developed based on the barycentric representation of Lagrange interpolating polynomials and the explicit barycentric weights for the GG points. The established GIMs and quadratures were optimized following the method of \cite{Elgindy2013}. The present numerical scheme leads to a reduction in the computational cost and time complexity while preserving the order of accuracy achieved by \cite{Elgindy2013}. The proposed HBGIM is a stable numerical scheme, which generally leads to well-conditioned systems. The numerical experiments confirm the stability, high-order accuracy, and efficiency of the proposed HBGIM and developed computational algorithms. The presented algorithms and numerical scheme provide easy yet strong numerical tools, which can be effectively carried out for the solution of a wide variety of problems. For instance, the current work laid the foundation of the exponentially-convergent numerical method of \cite{Elgindy2016a} for solving optimal control problems governed by parabolic distributed parameter systems, and was the crucial element in deriving a high-order adaptive spectral element  algorithm for solving general nonlinear optimal control problems exhibiting smooth/nonsmooth solutions using composite shifted Gegenbauer grids; cf. \cite{Elgindy2016b}. Other possible future directions may include the extension of the current work to handle problems in multiple-space dimensions and the development of sparse/banded integration matrices.
\section{Acknowledgments}
I would like to express my deepest gratitude to the editor for carefully handling the article, and the anonymous reviewers for their careful reading, constructive comments, and useful suggestions, which shaped the article into its final form.
\appendix
\section{Pseudocodes for the developed computational algorithms}
\label{appendix:CAP}
\begin{algorithm}[ht]
\renewcommand{\thealgorithm}{1}
\caption{Construction of the barycentric GIM for the GG set of integration points}
\label{sec:COTBGFTGSOP1:alg1matrix}
\begin{algorithmic}
\REQUIRE Positive integer $n$; the set of GG points and quadrature weights, $\left\{ {x_{n,i}^{(\alpha )},\varpi _{n,i}^{(\alpha )}} \right\}_{i = 0}^n$.\\
%GG nodes $x_{l,N_x,i}^{(\alpha)}, i = 0, \ldots, N_x$;\\
%%%\vspace{2mm}
\STATE $\xi _{n,i}^{(\alpha )} \leftarrow {( - 1)^i}{\mkern 1mu} \sin \left( {{{\cos }^{ - 1}}\left( {x_{n,i}^{(\alpha )}} \right)} \right)\sqrt {\varpi _{n,i}^{(\alpha )}} ,\quad i = 0, \ldots ,n.$
\STATE $N \leftarrow \left\lceil {(n - 1)/2} \right\rceil$.
\STATE Calculate the set of LG points and quadrature weights, $\left\{ {x_{N,i}^{(0.5)},\varpi _{N,i}^{(0.5)}} \right\}_{i = 0}^N$.\\
\STATE ${\mathbf{P}}_B^{(1)} \leftarrow {\mathbf{O}}$.
\FOR{$j = 0$ \TO $n$ }
\STATE{$\hat x_{N,k}^{(0.5)} \leftarrow \left( {\left( {x_{n,j}^{(\alpha )} + 1} \right)x_{N,k}^{(0.5)} + x_{n,j}^{(\alpha )} - 1} \right)/2,\quad k = 0, \ldots ,N$.}
\FOR{$k = 0$ \TO $N$ }
\STATE{${\mu _i} \leftarrow \xi _{n,i}^{(\alpha )}/\left( {\hat x_{N,k}^{(0.5)} - x_{n,i}^{(\alpha )}} \right),\quad i = 0, \ldots ,n.$\\
$\nu \leftarrow \sum\nolimits_{i = 0}^n {{\mu _i}}$.\\
${p_{B,j,i}} \leftarrow {p_{B,j,i}} + \varpi _{N,k}^{(0.5)}{\mu _i}/\nu ,\quad i = 0, \ldots ,n.$}
\ENDFOR
\STATE ${p_{B,j,i}} \leftarrow \left( {x_{n,j}^{(\alpha )} + 1} \right){p_{B,j,i}}/2,\quad i = 0, \ldots ,n.$
\ENDFOR
\RETURN{${\mathbf{P}}_B^{(1)}$}
\end{algorithmic}
\end{algorithm}
%%%%%%%%%%%%%%%%%%%%%%%%%%%%%%%%%%%%%%%%%%%%%%%%%%%%%%%%%%%%%%%%%%%%%%%%%%%%%%%%%%%%%%%%%%%%%%%%%%%%%%%%%%%%%%%%%%%%%%%%%%%
%%%%%%%%%%%%%%%%%%%%%%%%%%%%%%%%%%%%%%%%%%%%%%%%%%%%%%%%%%%%%%%%%%%%%%%%%%%%%%%%%%%%%%%%%%%%%%%%%%%%%%%%%%%%%%%%%%%%%%%%%%%
%%%%%%%%%%%%%%%%%%%%%%%%%%%%%%%%%%%%%%%%%%%%%%%%%%%%%%%%%%%%%%%%%%%%%%%%%%%%%%%%%%%%%%%%%%%%%%%%%%%%%%%%%%%%%%%%%%%%%%%%%%%
\begin{algorithm}[ht]
\renewcommand{\thealgorithm}{2}
\caption{Calculation of the barycentric Gegenbauer quadrature for the GG set of integration points}
\label{sec:COTBGQFTGSOP1:alg2matrix}
\begin{algorithmic}
\REQUIRE Positive integer $n$; the set of GG points and quadrature weights, $\left\{ {x_{n,i}^{(\alpha )},\varpi _{n,i}^{(\alpha )}} \right\}_{i = 0}^n$; real-valued function $f$.\\
%%%\vspace{2mm}
\STATE Construct ${\mathbf{P}}_B^{(1)}$ using Algorithm \ref{sec:COTBGFTGSOP1:alg1matrix}.\\
\STATE ${\mathbf{I}}_n^{(\alpha )} \leftarrow {\mathbf{P}}_B^{(1)} \cdot {\left( {f_{n,0}^{(\alpha )},f_{n,1}^{(\alpha )}, \ldots ,f_{n,n}^{(\alpha )}} \right)^T}$.
\RETURN{${\mathbf{I}}_n^{(\alpha )}$}
\end{algorithmic}
\end{algorithm}
%%%%%%%%%%%%%%%%%%%%%%%%%%%%%%%%%%%%%%%%%%%%%%%%%%%%%%%%%%%%%%%%%%%%%%%%%%%%%%%%%%%%%%%%%%%%%%%%%%%%%%%%%%%%%%%%%%%%%%%%%%%
%%%%%%%%%%%%%%%%%%%%%%%%%%%%%%%%%%%%%%%%%%%%%%%%%%%%%%%%%%%%%%%%%%%%%%%%%%%%%%%%%%%%%%%%%%%%%%%%%%%%%%%%%%%%%%%%%%%%%%%%%%%
%%%%%%%%%%%%%%%%%%%%%%%%%%%%%%%%%%%%%%%%%%%%%%%%%%%%%%%%%%%%%%%%%%%%%%%%%%%%%%%%%%%%%%%%%%%%%%%%%%%%%%%%%%%%%%%%%%%%%%%%%%%
\begin{algorithm}[ht]
\renewcommand{\thealgorithm}{3}
\caption{Modified construction of the barycentric GIM for the GG set of integration points}
\label{sec:MCOTBGFTGSOP1:alg1matrix}
\begin{algorithmic}
\REQUIRE Positive integer $n$; the set of GG points and quadrature weights, $\left\{ {x_{n,i}^{(\alpha )},\varpi _{n,i}^{(\alpha )}} \right\}_{i = 0}^n$; relatively small positive number $\varepsilon$.\\
%GG nodes $x_{l,N_x,i}^{(\alpha)}, i = 0, \ldots, N_x$;\\
%%%\vspace{2mm}
\STATE $\xi _{n,i}^{(\alpha )} \leftarrow {( - 1)^i}{\mkern 1mu} \sin \left( {{{\cos }^{ - 1}}\left( {x_{n,i}^{(\alpha )}} \right)} \right)\sqrt {\varpi _{n,i}^{(\alpha )}} ,\quad i = 0, \ldots ,n.$
\STATE $N \leftarrow \left\lceil {(n - 1)/2} \right\rceil$.
\STATE Calculate the set of LG points and quadrature weights, $\left\{ {x_{N,i}^{(0.5)},\varpi _{N,i}^{(0.5)}} \right\}_{i = 0}^N$.\\
\STATE ${\mathbf{P}}_B^{(1)} \leftarrow {\mathbf{O}}$
\STATE $L_{i} \leftarrow 1,\quad i = 0, \ldots, n$.
\FOR{$j = 0$ \TO $n$ }
\STATE{$\hat x_{N,k}^{(0.5)} \leftarrow \left( {\left( {x_{n,j}^{(\alpha )} + 1} \right)x_{N,k}^{(0.5)} + x_{n,j}^{(\alpha )} - 1} \right)/2,\quad k = 0, \ldots ,N$.}
\FOR{$k = 0$ \TO $N$ }
\STATE{${d_i} = \left( {\hat x_{N,k}^{(0.5)} - x_{n,i}^{(\alpha )}} \right),\quad i = 0, \ldots ,n.$\\
$I \leftarrow \left\{ {i:\;\left| {{d_i}} \right| > \varepsilon } \right\}.$\\
${\mu _l} \leftarrow \xi _{n,l}^{(\alpha )}/d_l, \quad l \in I.$\\
$\nu \leftarrow \sum\nolimits_{l = 0}^{\aleph _0} {{\mu _l}}$.\quad \COMMENT{$\aleph _0$ denotes the cardinal number of $I$.}\\
${L_l} = {\mu _l}/\nu ,\quad l \in I.$\\
${p_{B,j,i}} \leftarrow {p_{B,j,i}} + \varpi _{N,k}^{(0.5)}{L_i} ,\quad i = 0, \ldots ,n.$}\\
${L_l} = 1,\quad l \in I.$
\ENDFOR
\STATE ${p_{B,j,i}} \leftarrow \left( {x_{n,j}^{(\alpha )} + 1} \right){p_{B,j,i}}/2,\quad i = 0, \ldots ,n.$
\ENDFOR
\RETURN{${\mathbf{P}}_B^{(1)}$}
\end{algorithmic}
\end{algorithm}
%%%%%%%%%%%%%%%%%%%%%%%%%%%%%%%%%%%%%%%%%%%%%%%%%%%%%%%%%%%%%%%%%%%%%%%%%%%%%%%%%%%%%%%%%%%%%%%%%%%%%%%%%%%%%%%%%%%%%%%%%%%
%%%%%%%%%%%%%%%%%%%%%%%%%%%%%%%%%%%%%%%%%%%%%%%%%%%%%%%%%%%%%%%%%%%%%%%%%%%%%%%%%%%%%%%%%%%%%%%%%%%%%%%%%%%%%%%%%%%%%%%%%%%
%%%%%%%%%%%%%%%%%%%%%%%%%%%%%%%%%%%%%%%%%%%%%%%%%%%%%%%%%%%%%%%%%%%%%%%%%%%%%%%%%%%%%%%%%%%%%%%%%%%%%%%%%%%%%%%%%%%%%%%%%%%
\begin{algorithm}[ht]
\renewcommand{\thealgorithm}{4}
\caption{Testing the sufficient condition \eqref{eq:rarecase3} for the construction of the barycentric GIM for the GG set of integration points}
\label{sec:TTSCFTCOTBGFTGSOP:alg1matrix}
\begin{algorithmic}
\REQUIRE Positive integer $n$; the set of GG points, $\left\{ x_{n,i}^{(\alpha )},\right\}_{i = 0}^n$; relatively small positive number $\varepsilon$.\\
%GG nodes $x_{l,N_x,i}^{(\alpha)}, i = 0, \ldots, N_x$;\\
%%%\vspace{2mm}
\STATE $N \leftarrow \left\lceil {(n - 1)/2} \right\rceil$.
\STATE Calculate the set of LG points, $\left\{ x_{N,i}^{(0.5)} \right\}_{i = 0}^N$.\\
\FOR{$j = 0$ \TO $n$ }
\FOR{$k = 0$ \TO $N$ }
\STATE $J \leftarrow \left\{ {{J_i}:{J_i} = \left\{ \begin{array}{l}
1,\quad \left| {1 + x_{N,k}^{(0.5)} - 2\left( {1 + x_{n,i}^{(\alpha )}} \right)/\left( {1 + x_{n,j}^{(\alpha )}} \right)} \right| \le \varepsilon ,\\
0,\quad {\text{otherwise,}}
\end{array} \right.i = 0, \ldots ,n} \right\}.$
\IF[$\aleph _0$ denotes the cardinal number of $J$.]{$\sum\nolimits_{i = 0}^{{\aleph _0}} {{J_i}}  > 0$}
\STATE{\RETURN{``The test fails.''}}
\ENDIF
\ENDFOR
\ENDFOR
\RETURN{``The test succeeds.''}
\end{algorithmic}
\end{algorithm}
%%%%%%%%%%%%%%%%%%%%%%%%%%%%%%%%%%%%%%%%%%%%%%%%%%%%%%%%%%%%%%%%%%%%%%%%%%%%%%%%%%%%%%%%%%%%%%%%%%%%%%%%%%%%%%%%%%%%%%%%%%%
%%%%%%%%%%%%%%%%%%%%%%%%%%%%%%%%%%%%%%%%%%%%%%%%%%%%%%%%%%%%%%%%%%%%%%%%%%%%%%%%%%%%%%%%%%%%%%%%%%%%%%%%%%%%%%%%%%%%%%%%%%%
%%%%%%%%%%%%%%%%%%%%%%%%%%%%%%%%%%%%%%%%%%%%%%%%%%%%%%%%%%%%%%%%%%%%%%%%%%%%%%%%%%%%%%%%%%%%%%%%%%%%%%%%%%%%%%%%%%%%%%%%%%%
\begin{algorithm}[ht]
\renewcommand{\thealgorithm}{5}
\caption{Testing the sufficient condition for the construction of the optimal barycentric GIM for $m > m_{\max}$}
\label{sec:TTSKIMO1:alg1matrix}
\begin{algorithmic}
\REQUIRE Positive integers $m, l$; the set of GG points, $\left\{ x_{m,i}^{(\alpha_a )},\right\}_{i = 0}^m$; the set of integration points, $\{ x_{k}\}_{k = 0}^l$; relatively small positive number $\varepsilon$.\\
\STATE $M \leftarrow \left\lceil {(m - 1)/2} \right\rceil$.
\STATE Calculate the set of LG points, $\left\{ x_{M,s}^{(0.5)} \right\}_{s = 0}^M$.\\
\FOR{$i = 0$ \TO $m$ }
\FOR{$s = 0$ \TO $M$ }
\STATE $J \leftarrow \left\{ {{J_k}:{J_k} = \left\{ {\begin{array}{*{20}{l}}
{1,\quad \left| {x_{M,s}^{(0.5)} - ({1 - {x_k} + 2x_{m,i}^{({\alpha _a})}})/({{1 + {x_k}}})} \right| \le \varepsilon ,}\\
{0,\quad {\text{otherwise,}}}
\end{array}} \right.k = 0, \ldots ,l} \right\}.$
\IF[$\aleph _0$ denotes the cardinal number of $J$.]{$\sum\nolimits_{k = 0}^{{\aleph _0}} {{J_k}}  > 0$}
\STATE{\RETURN{``The test fails.''}}
\ENDIF
\ENDFOR
\ENDFOR
\RETURN{``The test succeeds.''}
\end{algorithmic}
\end{algorithm}
%%%%%%%%%%%%%%%%%%%%%%%%%%%%%%%%%%%%%%%%%%%%%%%%%%%%%%%%%%%%%%%%%%%%%%%%%%%%%%%%%%%%%%%%%%%%%%%%%%%%%%%%%%%%%%%%%%%%%%%%%%%
%%%%%%%%%%%%%%%%%%%%%%%%%%%%%%%%%%%%%%%%%%%%%%%%%%%%%%%%%%%%%%%%%%%%%%%%%%%%%%%%%%%%%%%%%%%%%%%%%%%%%%%%%%%%%%%%%%%%%%%%%%%
%%%%%%%%%%%%%%%%%%%%%%%%%%%%%%%%%%%%%%%%%%%%%%%%%%%%%%%%%%%%%%%%%%%%%%%%%%%%%%%%%%%%%%%%%%%%%%%%%%%%%%%%%%%%%%%%%%%%%%%%%%%
\begin{algorithm}[ht]
\renewcommand{\thealgorithm}{6}
\caption{Construction of a row barycentric GIM corresponding to the point $1$}
\label{sec:COARBGFTP1:alg1matrix}
\begin{algorithmic}
\REQUIRE Positive integer $n$; the set of GG points and quadrature weights, $\left\{ {x_{n,i}^{(\alpha )},\varpi _{n,i}^{(\alpha )}} \right\}_{i = 0}^n$.\\
\ENSURE ${\mathbf{P}}_B^{(1)}$ is not required.
\STATE $\xi _{n,i}^{(\alpha )} \leftarrow {( - 1)^i}{\mkern 1mu} \sin \left( {{{\cos }^{ - 1}}\left( {x_{n,i}^{(\alpha )}} \right)} \right)\sqrt {\varpi _{n,i}^{(\alpha )}} ,\quad i = 0, \ldots ,n.$
\STATE $N \leftarrow \left\lceil {(n - 1)/2} \right\rceil$.
\IF{${\text{mod}}(n,2) = 0\; \wedge \;\text{mod}(N,2) = 0$}
\STATE{$N \leftarrow N + 1.$}
\ENDIF
\STATE ${p_{B,n+1,i}} \leftarrow 0,\quad i = 0, \ldots ,n.$\\
Calculate the set of LG points and quadrature weights, $\left\{ {x_{N,i}^{(0.5)},\varpi _{N,i}^{(0.5)}} \right\}_{i = 0}^N$.\\
\FOR{$k = 0$ \TO $N$ }
\STATE{${\mu _i} \leftarrow \xi _{n,i}^{(\alpha )}/\left( {x_{N,k}^{(0.5)} - x_{n,i}^{(\alpha )}} \right),\quad i = 0, \ldots ,n.$\\
$\nu \leftarrow \sum\nolimits_{i = 0}^n {{\mu _i}}$.\\
${p_{B,n+1,i}} \leftarrow {p_{B,n+1,i}} + \varpi _{N,k}^{(0.5)}{\mu _i}/\nu ,\quad i = 0, \ldots ,n.$}
\ENDFOR
\RETURN{${\mathbf{P}}_{B,n+1}^{(1)}$}
\end{algorithmic}
\end{algorithm}
%%%%%%%%%%%%%%%%%%%%%%%%%%%%%%%%%%%%%%%%%%%%%%%%%%%%%%%%%%%%%%%%%%%%%%%%%%%%%%%%%%%%%%%%%%%%%%%%%%%%%%%%%%%%%%%%%%%%%%%%%%%
%%%%%%%%%%%%%%%%%%%%%%%%%%%%%%%%%%%%%%%%%%%%%%%%%%%%%%%%%%%%%%%%%%%%%%%%%%%%%%%%%%%%%%%%%%%%%%%%%%%%%%%%%%%%%%%%%%%%%%%%%%%
%%%%%%%%%%%%%%%%%%%%%%%%%%%%%%%%%%%%%%%%%%%%%%%%%%%%%%%%%%%%%%%%%%%%%%%%%%%%%%%%%%%%%%%%%%%%%%%%%%%%%%%%%%%%%%%%%%%%%%%%%%%
\begin{algorithm}[ht]
\renewcommand{\thealgorithm}{7}
\caption{Modified construction of a row barycentric GIM corresponding to the point $1$}
\label{sec:MCOARBGFTP1:alg1matrix}
\begin{algorithmic}
\REQUIRE Positive integer $n$; the set of GG points and barycentric weights, $\left\{ {x_{n,i}^{(\alpha )},\xi _{n,i}^{(\alpha )}} \right\}_{i = 0}^n$.\\
\ENSURE ${\mathbf{P}}_B^{(1)}$ is already calculated.
\STATE $N \leftarrow \left\lceil {(n - 1)/2} \right\rceil$.
\IF{${\text{mod}}(n,2) = 0\; \wedge \;\text{mod}(N,2) = 0$}
\STATE{$N \leftarrow N + 1.$}
\ENDIF
\STATE ${p_{B,n+1,i}} \leftarrow 0,\quad i = 0, \ldots ,n.$\\
Calculate the set of LG points and quadrature weights, $\left\{ {x_{N,i}^{(0.5)},\varpi _{N,i}^{(0.5)}} \right\}_{i = 0}^N$.\\
\FOR{$k = 0$ \TO $N$ }
\STATE{${\mu _i} \leftarrow \xi _{n,i}^{(\alpha )}/\left( {x_{N,k}^{(0.5)} - x_{n,i}^{(\alpha )}} \right),\quad i = 0, \ldots ,n.$\\
$\nu \leftarrow \sum\nolimits_{i = 0}^n {{\mu _i}}$.\\
${p_{B,n+1,i}} \leftarrow {p_{B,n+1,i}} + \varpi _{N,k}^{(0.5)}{\mu _i}/\nu ,\quad i = 0, \ldots ,n.$}
\ENDFOR
\RETURN{${\mathbf{P}}_{B,n+1}^{(1)}$}
\end{algorithmic}
\end{algorithm}
%%%%%%%%%%%%%%%%%%%%%%%%%%%%%%%%%%%%%%%%%%%%%%%%%%%%%%%%%%%%%%%%%%%%%%%%%%%%%%%%%%%%%%%%%%%%%%%%%%%%%%%%%%%%%%%%%%%%%%%%%%%
%%%%%%%%%%%%%%%%%%%%%%%%%%%%%%%%%%%%%%%%%%%%%%%%%%%%%%%%%%%%%%%%%%%%%%%%%%%%%%%%%%%%%%%%%%%%%%%%%%%%%%%%%%%%%%%%%%%%%%%%%%%
%%%%%%%%%%%%%%%%%%%%%%%%%%%%%%%%%%%%%%%%%%%%%%%%%%%%%%%%%%%%%%%%%%%%%%%%%%%%%%%%%%%%%%%%%%%%%%%%%%%%%%%%%%%%%%%%%%%%%%%%%%%
\begin{algorithm}[ht]
\renewcommand{\thealgorithm}{8}
\caption{Construction of the barycentric GIM for an arbitrary set of integration points}
\label{sec:COTBGFAASOP:alg1matrix}
\begin{algorithmic}
\REQUIRE Positive integers $n, m$; the set of GG points and quadrature weights, $\left\{ {x_{n,i}^{(\alpha )},\varpi _{n,i}^{(\alpha )}} \right\}_{i = 0}^n$; the set of the integration nodes $\{ {x_k}\} _{k = 0}^m$.\\
\STATE $\xi _{n,i}^{(\alpha )} \leftarrow {( - 1)^i}{\mkern 1mu} \sin \left( {{{\cos }^{ - 1}}\left( {x_{n,i}^{(\alpha )}} \right)} \right)\sqrt {\varpi _{n,i}^{(\alpha )}} ,\quad i = 0, \ldots ,n.$
\STATE $N \leftarrow \left\lceil {(n - 1)/2} \right\rceil$.
\STATE Calculate the set of LG points and quadrature weights, $\left\{ {x_{N,i}^{(0.5)},\varpi _{N,i}^{(0.5)}} \right\}_{i = 0}^N$.\\
\STATE ${\mathbf{P}}_B^{(1)} \leftarrow {\mathbf{O}}$.\quad \COMMENT{${\mathbf{O}} \in \mathbb{R}^{(m+1) \times (n+1)}$}\\
\FOR{$j = 0$ \TO $m$ }
\STATE{$\hat x_{N,k}^{(0.5)} \leftarrow \left( {\left( {x_{j} + 1} \right) x_{N,k}^{(0.5)} + x_{j} - 1} \right)/2,\quad k = 0, \ldots ,N$.}
\FOR{$k = 0$ \TO $N$ }
\STATE{${\mu _i} \leftarrow \xi _{n,i}^{(\alpha )}/\left( {\hat x_{N,k}^{(0.5)} - x_{n,i}^{(\alpha )}} \right),\quad i = 0, \ldots ,n.$\\
$\nu \leftarrow \sum\nolimits_{i = 0}^n {{\mu _i}}$.\\
${p_{B,j,i}} \leftarrow {p_{B,j,i}} + \varpi _{N,k}^{(0.5)}{\mu _i}/\nu ,\quad i = 0, \ldots ,n.$}
\ENDFOR
\STATE ${p_{B,j,i}} \leftarrow \left( {x_{j} + 1} \right){p_{B,j,i}}/2,\quad i = 0, \ldots ,n.$
\ENDFOR
\RETURN{${\mathbf{P}}_B^{(1)}$}
\end{algorithmic}
\end{algorithm}
%%%%%%%%%%%%%%%%%%%%%%%%%%%%%%%%%%%%%%%%%%%%%%%%%%%%%%%%%%%%%%%%%%%%%%%%%%%%%%%%%%%%%%%%%%%%%%%%%%%%%%%%%%%%%%%%%%%%%%%%%%%
%%%%%%%%%%%%%%%%%%%%%%%%%%%%%%%%%%%%%%%%%%%%%%%%%%%%%%%%%%%%%%%%%%%%%%%%%%%%%%%%%%%%%%%%%%%%%%%%%%%%%%%%%%%%%%%%%%%%%%%%%%%
%%%%%%%%%%%%%%%%%%%%%%%%%%%%%%%%%%%%%%%%%%%%%%%%%%%%%%%%%%%%%%%%%%%%%%%%%%%%%%%%%%%%%%%%%%%%%%%%%%%%%%%%%%%%%%%%%%%%%%%%%%%
\clearpage
\begin{algorithm}[ht]
\renewcommand{\thealgorithm}{9}
\caption{Construction of the optimal barycentric GIM for any non-symmetric set of integration points}
\label{sec1:alg:1}
\begin{algorithmic}
  \REQUIRE Positive integer numbers $n, m, m_{\max}$; positive real number $r \in [1, 2]$; the set of the integration nodes $\{ {x_k}\} _{k = 0}^n$; relatively small positive number $\varepsilon; \alpha_{a} \in \{0, 0.5\}; \alpha_b \in \{- 0.5 + {\varepsilon}, \alpha_{a}\}$.
	%\ENSURE $(n,m,\alpha,\alpha_{a}) \in \mathbb{F}_{OB}^{(m_{\max})}$.
\IF{$m > m_{\max}$} \STATE{$\alpha \leftarrow \alpha_{a}$.}
\IF{$m = n$} \STATE{Calculate the set of GG points and quadrature weights, $\left\{ {x_{n,i}^{(\alpha )},\varpi _{n,i}^{(\alpha )}} \right\}_{i = 0}^n$.\\
Calculate $\mathbf{P}_{B}^{(1)}$ using Algorithm \ref{sec:COTBGFAASOP:alg1matrix}.\\
$\mathbf{P}_{OB}^{(1)} \leftarrow \mathbf{P}_{B}^{(1)}$.}
\ELSE \STATE{$M \leftarrow \left\lceil {(m - 1)/2} \right\rceil$.\\
Calculate $\left\{ {x_{m,i}^{(\alpha )},\varpi _{m,i}^{(\alpha )},\xi _{m,i}^{(\alpha )}} \right\}_{i = 0}^m;\left\{ {x_{M,s}^{(0.5)},\varpi _{M,s}^{(0.5)}} \right\}_{s = 0}^M$.\\
\FOR{$k = 0$ \TO $n$ }
\STATE{$\hat x_{M,s}^{(0.5)} \leftarrow \left( {\left( {x_{k} + 1} \right)x_{M,s}^{(0.5)} + x_{k} - 1} \right)/2,\quad s = 0, \ldots ,M$.}
\FOR{$s = 0$ \TO $M$ }
\STATE{${\mu _i} \leftarrow \xi _{m,i}^{(\alpha )}/\left( {\hat x_{M,s}^{(0.5)} - x_{m,i}^{(\alpha )}} \right),\quad i = 0, \ldots ,m.$\\
$\nu \leftarrow \sum\nolimits_{i = 0}^m {{\mu _i}}$.\\
${p_{OB,k,i}} \leftarrow {p_{OB,k,i}} + \varpi _{M,s}^{(0.5)}{\mu _i}/\nu ,\quad i = 0, \ldots ,m.$}
\ENDFOR
\STATE ${p_{OB,k,i}} \leftarrow \left( {x_{k} + 1} \right){p_{OB,k,i}}/2,\quad i = 0, \ldots ,m.$
\ENDFOR
}
\ENDIF
\STATE $M \leftarrow \left\lceil {(m - 1)/2} \right\rceil$.\\
\IF{${\text{mod}}(m,2) = 0\; \wedge \;{\text{mod}}(M,2) = 0\; \wedge \;1 \in \{x_k\}_{k=0}^n$}
\STATE{$M \leftarrow M + 1$.}
\ENDIF
\STATE Calculate $\left\{ {x_{M,s}^{(0.5)},\varpi _{M,s}^{(0.5)}} \right\}_{s = 0}^M$.\\
\FOR{$k  = 0$ \TO $n$} \STATE{$\alpha _k^* \leftarrow \mathop {{\text{argmin}}}\limits_{-1/2< \alpha \le r}  \eta _{k,m}^2(\alpha )$.\\
\IF{$\alpha _k^* \in ( - 0.5, - 0.5 + \varepsilon )$} \STATE{$\alpha _k^* \leftarrow \alpha_b$.}
\ENDIF
\STATE Calculate $\left\{ {z_{m,k,i}^{(\alpha _k^*)},\varpi _{m,k,i}^{(\alpha _k^*)},\xi _{m,k,i}^{(\alpha _k^*)}} \right\}_{i = 0}^m$.\\
\STATE{$\hat x_{M,s}^{(0.5)} \leftarrow \left( {\left( {x_{k} + 1} \right)x_{M,s}^{(0.5)} + x_{k} - 1} \right)/2,\quad s = 0, \ldots ,M$.}
\FOR{$s = 0$ \TO $M$ }
\STATE{${\mu _i} \leftarrow \xi _{m,k,i}^{(\alpha _k^*)}/\left( {\hat x_{M,s}^{(0.5)} - z_{m,k,i}^{(\alpha _k^*)}} \right),\quad i = 0, \ldots ,m.$\\
$\nu \leftarrow \sum\nolimits_{i = 0}^m {{\mu _i}}$.\\
${p_{OB,k,i}} \leftarrow {p_{OB,k,i}} + \varpi _{M,s}^{(0.5)}{\mu _i}/\nu ,\quad i = 0, \ldots ,m.$}
\ENDFOR
\STATE ${p_{OB,k,i}} \leftarrow \left( {x_{k} + 1} \right){p_{OB,k,i}}/2,\quad i = 0, \ldots ,m.$
}
\ENDFOR
\ENDIF
\RETURN{$\mathbf{P}_{OB}^{(1)}$}.
\end{algorithmic}
\end{algorithm}
%%%%%%%%%%%%%%%%%%%%%%%%%%%%%%%%%%%%%%%%%%%%%%%%%%%%%%%%%%%%%%%%%%%%%%%%%%%%%%%%%%%%%%%%%%%%%%%%%%%%%%%%%%%%%%%%%%%%%%%%%%%
%%%%%%%%%%%%%%%%%%%%%%%%%%%%%%%%%%%%%%%%%%%%%%%%%%%%%%%%%%%%%%%%%%%%%%%%%%%%%%%%%%%%%%%%%%%%%%%%%%%%%%%%%%%%%%%%%%%%%%%%%%%
%%%%%%%%%%%%%%%%%%%%%%%%%%%%%%%%%%%%%%%%%%%%%%%%%%%%%%%%%%%%%%%%%%%%%%%%%%%%%%%%%%%%%%%%%%%%%%%%%%%%%%%%%%%%%%%%%%%%%%%%%%%
\clearpage
\begin{algorithm}[ht]
\renewcommand{\thealgorithm}{10}
\caption{Construction of the optimal barycentric GIM for any symmetric set of integration points}
\label{sec1:alg:2}
\begin{algorithmic}
  \REQUIRE Positive integer numbers $n, m_{\max}$; positive even integer $m$; positive real number $r \in [1, 2]$; the set of the integration nodes $\{ {x_k}\} _{k = 0}^n$; relatively small positive number $\varepsilon; \alpha_{a} \in \{0, 0.5\}; \alpha_b \in \{- 0.5 + {\varepsilon}, \alpha_{a}\}$.
	%\ENSURE $(n,m,\alpha,\alpha_{a}) \in \mathbb{F}_{OB}^{(m_{\max})}$.
\IF{$m > m_{\max}$} \STATE{$\alpha \leftarrow \alpha_{a}$.}
\IF{$m = n$} \STATE{Calculate the set of GG points and quadrature weights, $\left\{ {x_{n,i}^{(\alpha )},\varpi _{n,i}^{(\alpha )}} \right\}_{i = 0}^n$.\\
Calculate $\mathbf{P}_{B}^{(1)}$ using Algorithm \ref{sec:COTBGFAASOP:alg1matrix}.\\
$\mathbf{P}_{OB}^{(1)} \leftarrow \mathbf{P}_{B}^{(1)}$.}
\ELSE \STATE{$M \leftarrow \left\lceil {(m - 1)/2} \right\rceil$.\\
Calculate $\left\{ {x_{m,i}^{(\alpha )},\varpi _{m,i}^{(\alpha )},\xi _{m,i}^{(\alpha )}} \right\}_{i = 0}^m;\left\{ {x_{M,s}^{(0.5)},\varpi _{M,s}^{(0.5)}} \right\}_{s = 0}^M$.\\
\FOR{$k = 0$ \TO $n$ }
\STATE{$\hat x_{M,s}^{(0.5)} \leftarrow \left( {\left( {x_{k} + 1} \right)x_{M,s}^{(0.5)} + x_{k} - 1} \right)/2,\quad s = 0, \ldots ,M$.}
\FOR{$s = 0$ \TO $M$ }
\STATE{${\mu _i} \leftarrow \xi _{m,i}^{(\alpha )}/\left( {\hat x_{M,s}^{(0.5)} - x_{m,i}^{(\alpha )}} \right),\quad i = 0, \ldots ,m.$\\
$\nu \leftarrow \sum\nolimits_{i = 0}^m {{\mu _i}}$.\\
${p_{OB,k,i}} \leftarrow {p_{OB,k,i}} + \varpi _{M,s}^{(0.5)}{\mu _i}/\nu ,\quad i = 0, \ldots ,m.$}
\ENDFOR
\STATE ${p_{OB,k,i}} \leftarrow \left( {x_{k} + 1} \right){p_{OB,k,i}}/2,\quad i = 0, \ldots ,m.$
\ENDFOR
}
\ENDIF
\ELSE
\STATE $M \leftarrow \left\lceil {(m - 1)/2} \right\rceil$.\\
\IF{${\text{mod}}(M,2) = 0\; \wedge \;1 \in \{x_k\}_{k=0}^n$}
\STATE{$M \leftarrow M + 1$.}
\ENDIF
\STATE Calculate $\left\{ {x_{M,s}^{(0.5)},\varpi _{M,s}^{(0.5)}} \right\}_{s = 0}^M$.\\
\FOR{$k = 0$ \TO $n$}
\STATE{\IF{$k \le \left\lfloor {n/2} \right\rfloor$} \STATE{$\alpha _k^* \leftarrow \mathop {{\text{argmin}}}\limits_{-1/2< \alpha \le r}  \eta _{k,m}^2(\alpha )$.\\
\IF{$\alpha _k^* \in ( - 0.5, - 0.5 + \varepsilon )$} \STATE{$\alpha _k^* \leftarrow \alpha_b$.}
\ENDIF
\STATE Calculate $\left\{ {z_{m,k,i}^{(\alpha _k^*)},\varpi _{m,k,i}^{(\alpha _k^*)},\xi _{m,k,i}^{(\alpha _k^*)}} \right\}_{i = 0}^m$.\\
$\left\{ {z_{m,n - k,i}^{(\alpha _{n-k}^*)},\xi _{m,n - k,i}^{(\alpha _{n-k}^*)}} \right\}_{i = 0}^m \leftarrow \left\{ {z_{m,k,i}^{(\alpha _k^*)},\xi _{m,k,i}^{(\alpha _k^*)}} \right\}_{i = 0}^m.$
}
\ENDIF
\STATE{$\hat x_{M,s}^{(0.5)} \leftarrow \left( {\left( {x_{k} + 1} \right)x_{M,s}^{(0.5)} + x_{k} - 1} \right)/2,\quad s = 0, \ldots ,M$.}
\FOR{$s = 0$ \TO $M$ }
\STATE{${\mu _i} \leftarrow \xi _{m,k,i}^{(\alpha _k^*)}/\left( {\hat x_{M,s}^{(0.5)} - z_{m,k,i}^{(\alpha _k^*)}} \right),\quad i = 0, \ldots ,m.$\\
$\nu \leftarrow \sum\nolimits_{i = 0}^m {{\mu _i}}$.\\
${p_{OB,k,i}} \leftarrow {p_{OB,k,i}} + \varpi _{M,s}^{(0.5)}{\mu _i}/\nu ,\quad i = 0, \ldots ,m.$}
\ENDFOR
\STATE ${p_{OB,k,i}} \leftarrow \left( {x_{k} + 1} \right){p_{OB,k,i}}/2,\quad i = 0, \ldots ,m.$
}
\ENDFOR
\ENDIF
\RETURN{$\mathbf{P}_{OB}^{(1)}$}.
\end{algorithmic}
\end{algorithm}
\clearpage
%%%\appendixpage
%%\section{Preliminaries}
%%\label{sec1}
%%
%% ---------------------------------------------
%% References
%%
%%%\bibliographystyle{model1-num-names}
%\bibliographystyle{elsarticle-harv}
%\bibliography{Bib}

%% ---------------------------------------------
%%
\end{document}